\documentclass[11pt]{article}
\usepackage{amssymb,amsmath,bm,amsthm,mathrsfs,citesort,setspace}
\usepackage[color,notcite,notref]{showkeys}
\usepackage{pdfsync}

\newcommand{\mathd}{\mathrm{d}}
\newcommand{\nin}{\not\in}
\newcommand{\tmdummy}{$\mbox{}$}
\newcommand{\tmem}[1]{{\em #1\/}}
\newcommand{\tmmathbf}[1]{\ensuremath{\boldsymbol{#1}}}
\newcommand{\tmop}[1]{\ensuremath{\operatorname{#1}}}
\newcommand{\tmtextbf}[1]{{\bfseries{#1}}}
\newcommand{\tmtextit}[1]{{\itshape{#1}}}
\newcommand{\tmtextrm}[1]{{\rmfamily{#1}}}
\newcommand{\tmtexttt}[1]{{\ttfamily{#1}}}
\newenvironment{enumerateroman}{\begin{enumerate}}{\end{enumerate}}

\theoremstyle{plain}
\newtheorem{theorem}{Theorem}[section]
\newtheorem{proposition}[theorem]{Proposition}
\newtheorem{corollary}[theorem]{Corollary}
\newtheorem{lemma}[theorem]{Lemma}

\theoremstyle{remark}
\newtheorem{remark}[theorem]{Remark}

\numberwithin{equation}{section}

\title{Charged polymers in the attractive regime: a first-order 
	transition from Brownian scaling to
	four-point localization}
\author{Yueyun Hu\\Universit\'e Paris 13 
	\and   Davar Khoshnevisan\thanks{Research supported by
	NSF grants DMS-0706728 and DMS-1006903}\\University of Utah
	\and Marc Wouts\\Universit\'e Paris 13}
\date{November 2, 2010}

\begin{document}\onehalfspacing

\maketitle

\begin{abstract}
  We study a quenched charged-polymer model, introduced by Garel and Orland
  in 1988, that reproduces the folding/unfolding transition of biopolymers. 
  We prove that, below the critical inverse temperature, 
  the polymer is delocalized in the sense
  that: (1) The rescaled trajectory of the polymer converges to the 
  Brownian path; and (2) The partition function remains bounded. 
  
  At the critical inverse temperature, we show that the maximum
  time spent at points jumps discontinuously from $0$ to a 
  positive fraction of the number of monomers, in the limit as the number of monomers
  tends to infinity.

  Finally, when the critical inverse temperature is large,
  we prove that the polymer collapses in the sense
  that a large fraction of its monomers live on four adjacent 
  positions, and its diameter grows only logarithmically
  with the number of the monomers. 
  
  Our methods also provide some insight into  the annealed phase
  transition and at the transition due to a pulling force; both phase
  transitions are shown to be discontinuous.\\

  {\noindent}\tmtextit{Keywords:} Charged polymers, quenched measure, annealed
  measure, localization-delocalization transition, first-order phase
  transition.\\
  
  {\noindent}\tmtextit{{\noindent}AMS 2000 subject classification:} Primary:
  60K35; Secondary: 60K37.
\end{abstract}

\section{Introduction and main results}

\subsection{The charged polymer model}

We consider a polymer model introduced by  Garel and Orland in {\cite{GO88b}}
for modeling the trajectory of biological proteins made of hydrophobic
monomers. Let $\{q_i \}_{i = 0}^{\infty}$ be i.i.d.\
real variables and $\{S_i\}_{i = 0}^{\infty}$ an independent simple 
random walk  on $\mathbf{Z}^d$ with $S_0 = 0$. 
Both stochastic processes exist on a common probability
space $(\Omega\,,\mathcal{F},\mathrm{P})$. 

Given a realization of $q$ and $S$, we consider
\begin{align}
	  Q_N^x & := \sum_{0 \leqslant i < N} q_i  \mathbf{1}_{\{S_i = x\}},
	  	\quad\text{and}\label{eq:Q}\\
	  H_N & := \sum_{x \in \mathbf{Z}^d} \left( Q_N^x \right)^2.
	  	\label{eq:HN}
\end{align}
We think of the $q_i$'s as {\em charges},
$Q_N^x$ as the {\em total charge} at position $x \in
\mathbf{Z}^d$, and $H_N$ as
the {\em energy} of the polymer. In this way, we see that $Q^x_N$
and $H_N$ in fact define functions of the trajectory $S$ of the walk.
Therefore, we might occasionally refer to them respectively
as $Q^x_N(S)$ and $H_N(S)$, as well.

For all $\beta\in\mathbf{R}$ [{\em inverse temperature}]
and $N\geqslant 1$ [the number of monomers] consider the quenched
probability measure $\mathrm{P}_N^{\beta}$,
\begin{equation}
	\mathrm{P}_N^{\beta} (A)  :=  \frac{1}{Z_N (\beta)} \mathrm{E} \left[
	\left. \mathbf{1}_A \exp \left( \frac{\beta}{N} H_N \right) \right| q_0,
	q_1, \ldots, q_{N - 1} \right],
\end{equation}
where $Z_N(\beta)$ [the {\em partition function}] is defined so that
$\mathrm{P}_N^\beta$ is indeed a probability measure; that is,
\begin{equation}
	Z_N (\beta)  :=  \mathrm{E} \left[ \left. \exp \left( \frac{\beta}{N} H_N
	\right) \right| q_0, q_1, \ldots, q_{N - 1} \right] .
\end{equation}

We can write the energy in the following equivalent form:
\begin{equation}
	H_N := 2 \hat{H}_N + \sum_{0 \leqslant i < N} q_i^2,
\end{equation}
where
\begin{equation}
	\hat{H}_N  := \mathop{\sum\sum}\limits_{
	0 \leqslant i < j < N} q_i q_j  \mathbf{1}_{\{S_i =
	S_j \}}.\label{eq:Hhat}
\end{equation}
Therefore, if we define
the quenched measure $\widehat{\mathrm{P}}_N^{\beta}$ as
we did $\mathrm{P}_N^\beta$ but with
$\hat{H}_N$ in place of $H_N$, then
$\widehat{\mathrm{P}}_N^\beta=\mathrm{P}_N^{\beta / 2}$.
Thus, the analyses of $\mathrm{P}_N^\beta$ and 
$\widehat{\mathrm{P}}_N^\beta$ are the same, but one has to remember
to halve/double the parameter $\beta$ in order to understand one
in terms of the other.

In our model, like charges attract when $\beta > 0$ .
This accounts for
the hydrophobic properties of monomers immersed in water {\cite{GO88b}}. 
And the scaling $H_N / N$ was introduced also in \cite{GO88b} in
order to compensate for the absence of hard-core repulsion.
It will also follow from Lemma \ref{lem:subadd} below that
 this scaling makes the energy subadditive [or {\em extensive}].
 The fact that charges interact only when they are at exactly the 
 same position is said to account for the {\tmem{screening effect}}:
 When a polymer is immersed in water, its charges are
 surrounded by oppositely-charged free molecules of the solvent. 
 
Garel and Orland {\cite{GO88a,GO88b}} introduced the 
charged-polymer model
in order to better understand the transition, in biopolymers, 
from a swollen state to a folded state.
In {\cite{GO88a}} the authors perform a mean-field analysis of a 
model with independent, Gaussian
interactions between monomers pairs. And in {\cite{GO88b}} they
introduce [a generalization of] $\hat{H}_N$ in order to model different
possible attractive/repulsive forces between different monomers such 
as amino acids in proteins or the base-pairs in the 
RNA.\footnote{The energy in {\cite{GO88b}} corresponds to ours 
when their $M = 1$.}

When the reference random
walk $\{S_i\}_{i=0}^\infty$ is replaced by a walk on a simplex with $d$ points,
Garel and Orland \cite{GO88b}
find a continuous phase transition from a folded to an unfolded state as
the temperature increases. And, for a continuous version of the charged-polymer
model, they find that a similar continuous phase transition holds at an 
explicit temperature. In another
paper {\cite{GO88a}}, however, Garel and Orland mention
that the phase transition in
biopolymers is expected to be discontinuous. Among other things,
the results of our paper confirm their prediction in the present charged-polymer model.

The physics literature contains also the analyses of several 
seemingly-similar models that are not
equivalent to ours mainly because in those models 
like charges repel \cite{KK91,DGH92,DH94,GK97}.

In the last few years the mathematics of polymer measures
has also grown considerably \cite{vdHK01,CSY04,Gia07,dO09}.
However, it appears that little is
known about our model. We are aware only of Chapter 8 of \cite{dO09}
on the annealed measure in the repelling regime $\beta \leqslant 0$, and that 
result holds for a different scaling of the energy 
[for which the polymer is completely localized.]

We are aware also of some recent works on the energy $\hat{H}_N$
itself:  In \cite{C08}, limit theorems for $\hat{H}_N$ are established;
it was shown in \cite{CK09} that the distribution of
$\hat{H}_N$ is comparable to the random walk in random scenery
as $N$ tends to infinity, see also \cite{HK10}; and large deviations for
$\hat{H}_N$ were established in \cite{A08,A09}.

Let us conclude this introduction with a brief outline
of the paper: In the remainder of this section we
present our main results on the model.
Those results range from a characterization of the
delocalized phase to a description of the discontinuous phase transition,
and finally to large-$\beta$ asymptotics. We also emphasize some differences
between the quenched and annealed measures, and describe the effect of a
pulling force. Proofs of the various assertions are relegated to
Section \ref{sec:proofs}. Finally, we include some
basic facts about the local times of the simple random walk
in the appendix.

\subsection{The delocalized phase}

Unless it is stated to the contrary,
we assume that $\mathrm{E} q_0 = 0$, $\tmop{Var} q_0
= 1$, and that the charges are \emph{subgaussian};
that is, $\kappa<\infty$, where
\begin{equation}
	\kappa := \inf\left\{ c\in(-\infty\,,\infty]:\
	\mathrm{E} {\rm e}^{tq_0}  \leqslant  
	{\rm e}^{c t^2/2} \quad\text{for all }\quad t \in \mathbf{R}
	\right\}.
  	\label{eq:kappa}
\end{equation}
We have $\kappa\geqslant 1$ as long as $q_0$ has a finite moment
generating function near zero and $\mathrm{E}q_0=0$. 
And  $\kappa = 1$ both when the $q_i$'s have the
Rademacher distribution
[$\mathrm{P} \{q_0 = \pm 1\} = 1 / 2$] and when they have a
standard normal distribution.

Now we introduce
\begin{equation}
	\mathscr{D} := \left\{ \beta \in \mathbf{R} :\ Z_N (\beta) 
	\overset{\mathrm{P}}{\longrightarrow}{\rm e}^{\beta}
	\quad\text{as $N\to\infty$}
	\right\},
\end{equation}
where ``$\overset{\mathrm{P}}{\longrightarrow}$'' denotes convergence in
probability. As is customary,
we call
\begin{equation}\label{eq:LTx}
	L_N^x := \sum_{i = 0}^{N - 1} \tmmathbf{1}_{\{S_i = x\}}
\end{equation}
the \emph{local time} of $\{S_i \}_{i = 0}^{N - 1}$ at $x$, and 
define
\begin{equation}\label{eq:LTstar}
	L_N^{\star} := \max_{x \in \mathbf{Z}^d} L_N^x
\end{equation}
to be maximum local time.

The next theorem tells us that the set
$\mathscr{D}$ characterizes the region of $\beta$ for which the trajectory of
the polymer is [asymptotically] indistinguishable from that of a random walk.
In other words, the polymer is {\tmem{delocalized}}
when $\beta\in\mathscr{D}$ and $N$ is large.

\begin{theorem}\label{thm:D}
	If $\mathrm{E} q_0 = 0$, 
	$\tmop{Var} q_0 = 1$, and $\kappa < \infty$, then:
	\begin{enumerate}
		\item $\mathscr{D}$ is an interval that contains $(- \infty\,, 1 / \kappa)$.    
		\item $\beta \in \mathscr{D}$ if and only if
			 for all $\varepsilon > 0$,
			\begin{equation}
				\mathrm{P}_N^{\beta} \{L^{\star}_N \leqslant
				\varepsilon N\}\overset{\mathrm{P}}{\longrightarrow} 1
				\qquad\text{as $N\to\infty$.}
			\end{equation}
		\item $\beta \in \mathscr{D}$ if and only if:
		\begin{equation}
     			\left\| \mathrm{P}_N^{\beta}  - \mathrm{P} \left[\,\cdot \left|q_0\,, \ldots, q_{N - 1}\right. \right] \right\|_{\tmop{TV}}
			\overset{\mathrm{P}}{\longrightarrow}  0 
				\qquad\text{as $N\to\infty$,}
    		\end{equation}
		where $\left\| \mu -\nu \right\|_{\tmop{TV}}:=\sup_A |\mu(A)-\nu(A)|$ is the total variation distance. 
 	 \end{enumerate}
\end{theorem}

In order to describe a consequence of Theorem 
\ref{thm:D}, let $N \geqslant 1$ be an integer,
and consider the stochastic process
$\mathscr{S}_N$ defined by
\begin{equation}
	\mathscr{S}_N (t) := (Nt - [Nt]) \left( \frac{S_{[Nt] + 1} -
	S_{[Nt]}}{\sqrt{N}} \right) + \frac{S_{[Nt]}}{\sqrt{N}} 
	\qquad(0\leqslant t\leqslant 1).
\end{equation}
$\mathscr{S}_N$ is defined uniquely as the piecewise-linear
function that takes the values $S_k / \sqrt{N}$ at $t=k/N$ for
all integers $k = 0, \ldots, N$. Now we can mention the
consequence of Theorem \ref{thm:D}.

\begin{corollary}\label{cor:BM}
	If  $\mathrm{E} q_0 = 0$, $\tmop{Var} q_0 = 1$, and $\kappa
	< \infty$, then for all $\beta \in \mathscr{D}$ and 
	$\Phi : C ([0\,, 1]) \to \mathbf{R}$ bounded and continuous,
	\begin{eqnarray}
		\mathrm{E}^{\beta}_N \left[ \Phi \left( \mathscr{S}_N \right) \right] &
		\underset{N \rightarrow \infty}{\overset{\mathrm{P}}{\longrightarrow}} &
		\mathrm{E} \left[ \Phi \left( \mathscr{B} \right) \right], 
	\end{eqnarray}
	where $\mathscr{B}$ denotes $d$-dimensional Brownian motion.
\end{corollary}

\begin{remark}
	Even though $\beta\in\mathscr{D}$ if and only if
	$\mathrm{P}_N^{\beta} \{L^{\star}_N <
	\varepsilon N\} \rightarrow 1$ in probability,  one can say more about
	the rate of this convergence when
	$\beta$ in the interior of $\mathscr{D}$. Indeed, suppose $\beta$
	lies in the interior of $\mathscr{D}$. It follows from part 1 of
	Theorem \ref{thm:D} that $q\beta\in\mathscr{D}$ for some $q>1$.
	Let $p$ denote the conjugate to $q$; that is, $p^{-1}+q^{-1}=1$. Then
	H\"older's inequality implies that
	\begin{equation}
		\mathrm{P}_N^{\beta} \left\{
		L^{\star}_N \geqslant \varepsilon N\right\} \leqslant
		\left[ \mathrm{P} \left\{
		L^{\star}_N \geqslant \varepsilon N\right\}\right]^{1 / p}
		\cdot \frac{\left[Z_N (q
		\beta)\right]^{1 / q}}{Z_N (\beta)}.
	\end{equation}
	The fraction of the $Z_N$'s goes to one in probability since both
	$\beta$ and $q\beta$ are in $\mathscr{D}$. Therefore,
	it follows from Lemma \ref{lem:LTstar} below 
	that $\mathrm{P}_N^\beta\{L^{\star}_N
	< \varepsilon N\}\to 1$, in probability, exponentially fast,
	as long as $\beta$ lies in the interior of $\mathscr{D}$.\qed
\end{remark}

\subsection{A first-order phase transition}

We show, in Lemma \ref{lem:subadd} below,
that the normalized energy $H_N / N$ is
subadditive. And it will follow from that fact
that the {\em free energy} $\digamma$ exists when the second moment of
the charge distribution is finite. More precisely, we have the following.

\begin{proposition}\label{prop:F}
	If $\mathrm{E} (q_0^2) < \infty$, then for all $\beta \in\mathbf{R}$,
	\begin{eqnarray}\label{eq:F}
		\digamma (\beta) := \lim_{N\to\infty} 
		\frac{1}{N} \ln Z_N \left( \beta \right)		
	\end{eqnarray}
	exists a.s.\ and in $L^1 ( \mathrm{P})$, and 
	$\digamma(\beta)$ is nonrandom. The function
	$\mathbf{R} \ni \beta \rightarrow \digamma (\beta)$ is nonnegative,
	nondecreasing, and convex with $\digamma (0) = 0$.
\end{proposition}

Define
the \emph{critical inverse temperature},
\begin{equation}\label{eq:bc}
	\beta_c := \sup \mathscr{D}.  
\end{equation}
Clearly, $\digamma(\beta)=0$ whenever $\beta \leqslant \beta_c$.
We now wish to know whether or not the converse is true. 

Our next theorem shows that a first-order phase transition
occurs at $\beta_c$, and that the maximal fraction $L^{\star}_N / N$ 
of monomers on a single site jumps discontinuously from 
$0$ to a quantity that is at least $1 / (2\kappa \beta_c)>0$.
It might help to recall that convex functions have right derivatives
everywhere.

\begin{theorem}\label{thm:fo}
	If $\mathrm{E} q_0 = 0$,  $\tmop{Var} q_0 = 1$, and $\kappa
	< \infty$, then $\digamma(\beta_c)=0$,
	whereas $\digamma(\beta)>0$ for all $\beta>\beta_c$.
	Moreover, there is a first-order phase transition at $\beta_c$; i.e.,
	\begin{equation}\label{eq:fo}
		\lim_{\beta\downarrow\beta_c}
		\frac{\digamma (\beta)}{\beta-\beta_c} \in(0\,,\infty).
	\end{equation}
	Furthermore, if $\beta>\beta_c$, then for all $\varepsilon>0$,
	\begin{equation}\label{eq:prop:min}
		\mathrm{P}_N^{\beta} \left\{ \frac{L^{\star}_N}{N}
		\geqslant \frac{1-\varepsilon}{\beta} \max\left(
		\digamma (\beta) \,, \frac{1}{2\kappa}\right)\right\}  
		\overset{\mathrm{P}}{\longrightarrow} 1
		\quad\text{as $N\to\infty$}.
	\end{equation}
\end{theorem}

\subsection{The folded phase}

When the inverse temperature $\beta$ is large, 
the polymer measure concentrates
on the configurations with high energy. In dimensions $d\geqslant 2$ we
will compute the [quenched] maximum of $H_N$.
It turns out that that maximum is realized
when the walk is concentrated on four points that define a square. 

Recall that  $a^+ := a\vee 0$ and $a^- := (-a)^+$
for all $a \in \mathbf{R}$. When $Z$ is a random variable
and $\varepsilon\in\{-\,,+\}$
we always write $\mathrm{E}Z^\varepsilon$ as shorthand for
$\mathrm{E}(Z^\varepsilon)$ [and never for
$(\mathrm{E}Z)^\varepsilon$].

\begin{proposition}\label{prop:maxH:bc}
	If $d \geqslant 2$, then for all $\beta\in\mathbf{R}$,
	\begin{equation}
		\liminf_{N\to\infty}\frac1N\ln Z_N(\beta) 
		\geqslant \left[ \frac{(\mathrm{E} q_0^+)^2 +
		(\mathrm{E} q_0^-)^2}{2}\right] \beta - \ln (2 d)
		\qquad\text{a.s.}
	\end{equation}
	Consequently, under the assumptions
	of Theorem \ref{thm:D} [that $\mathrm{E}q_0=0$, $\tmop{Var}q_0=1$,
	and $\kappa<\infty$], the critical inverse temperature satisfies
	\begin{equation}
		\beta_c  \leqslant  \frac{2 \ln (2 d)}{(\mathrm{E} q_0^+)^2 
		+ (\mathrm{E} q_0^-)^2} .
	\end{equation}
\end{proposition}

We emphasize that, in the case that $\mathrm{E}|q_0|=\infty$, the
preceding proposition tells us that 
 $\digamma (\beta) = \infty$ a.s.\ for all $\beta > 0$.
That proposition also tells us that
$\beta_c \leqslant 4\ln (2 d)$ when $q_0$ has
 the Rademacher distribution [i.e., $q_0=\pm 1$ with probability $1/2$ each]
 and $\beta_c\leqslant 2\pi\ln(2d)$ when $q_0$ has a standard normal distribution.

In order to prepare for our next results 
we first define the following quantities:
\begin{align}
	\gamma & :=  \min_{\varepsilon\in\{-,+\}} 
		( \mathrm{E} q_0^{\varepsilon})^2;
		\label{eq:gamma}\\
	\lambda & := \min_{\varepsilon, \varepsilon'\in\{-,+\}} \mathrm{E}
		\left[ \min \left( (
		\mathrm{E} q_0^{\varepsilon}) q_0^{\varepsilon}\,, ( \mathrm{E}
		q_1^{\varepsilon'}) q_1^{\varepsilon'} \right)\right];
		\quad\text{and}
		\label{eq:lambda}\\
	\beta_{\alpha} & := \ln (2 d) \cdot \left[ \frac{8}{(1 - \alpha)
		\gamma} \vee \frac{4}{\lambda} \right]
		\qquad(0<\alpha<1).\label{eq:betaa}
\end{align}
We are interested mainly in $\beta_\alpha$
[$\beta_\alpha$ should not be confused with the critical
inverse temperature $\beta_c$.]

It is possible to check that when $q_0$ has a symmetric distribution
[i.e., $q_0$ and $-q_0$ have the same law],
\begin{equation}\begin{split}
	\gamma &=(\mathrm{E}q_0^+)^2=\left(\int_0^\infty
		\mathrm{P}\{q_0>z\}\, {\rm d}z\right)^2,\\
\lambda &=\sqrt\gamma\cdot\mathrm{E}( q_0^+\wedge q_1^+)=
	\sqrt\gamma\cdot
	\int_0^\infty\left( \mathrm{P}\{q_0>z\}\right)^2\,{\rm d}z,\\
\beta_\alpha&=\frac{4\ln(2d)}{\sqrt\gamma}
	\cdot \left[ \frac{2}{(1-\alpha)\sqrt\gamma}
	\vee \frac{1}{\mathrm{E}(q_0^+\wedge q_1^+)}\right].
\end{split}\end{equation}

Thus, for example, $\gamma=1/4$, 
$\lambda=1/8$, and $\beta_\alpha=32\ln(2d)/(1-\alpha)$
when $q_0$ has the Rademacher distribution [$q_0=\pm1$
with probability $1/2$ each].

In addition to the preceding constants, we will need
some notation:
We say that  ``\emph{$U$ is a unit square}'' if we can write 
$U =\{x_1, \ldots, x_4 \}$ as a collection of four points that satisfy
$\|x_2 - x_1 \|=\|x_3 - x_2 \|=\|x_4 - x_3 \|=\|x_1 - x_4 \|= 1$. 

Also, for $0<\alpha<1$ we define the event $\mathcal{S}_\alpha$,
\begin{equation}
	\mathcal{S}_{\alpha}  :=  \left\{ \begin{array}{c}
	\text{There exists a unique unit square } U \subset \mathbf{Z}^d
	\text{ such that}\\\displaystyle
	\sum_{1\leqslant i < N : S_i \nin U} |q_i | \leqslant
	\frac{1-\alpha}{2} \sum_{1\leqslant i < N} |q_i |
	\end{array} \right\} . \label{eq:S}
\end{equation}
In other words, the event $\mathcal{S}_{\alpha}$ is realized 
exactly when there exists a unique unit square $U$ such that
the sum of the absolute charges not on $U$
is at most  $(1-\alpha) / 2$ times the total absolute
charge of the polymer.

\begin{theorem}[The four points]\label{thm:square}
	Assume  $d \geqslant 2$. Then for all $\delta
	> 0$, there is $c_{\delta}\in(0\,,\infty)$ such that
	for every $N \geqslant 1$ and $\beta \in \mathbf{R}$,
	\begin{equation*}
		\mathrm{P} \left\{ \mathrm{P}_N^{\beta}
		\left( \mathcal{S}_{\alpha} \right)
		\geqslant 1 - \exp \left( N \ln \left( 2 d \right)  \left[ 1 -
		\frac{\beta}{(1 + \delta) \beta_{\alpha}} \right]\right) \right\}
		\geqslant  1 - \exp \left( - c_{\delta} N \right) .
		\end{equation*}
\end{theorem}

Our result is limited to $d \geqslant 2$, since 
this is the minimal dimension in which we can
consider a square. But other results are also sometimes possible.
For example, if $S$ is replaced by the lazy random walk, then
one can adapt the present methods to prove the existence of
two adjacent points that bear most of the available charge
provided that $\beta$ is large enough. And the latter assertion
is valid for every $d \geqslant 1$.

In the usual scaling $\beta H_N / N$, Theorem \ref{thm:square} shows that the
polymer is localized for any $\beta > \beta_{\alpha}$. But the latter theorem
yields a pointwise estimate in $\beta$.
It is instructive to also
consider the scaling in which $\beta = bN$ for some $b>0$.
That is the case in which $\beta$ is proportional to $N$ instead of
being a constant. In that case, $\mathcal{S}_{\alpha}$  continues to be
a typical event when $\alpha = 1 - 8 \ln (2 d) / ((1 + 2 \delta) b \gamma N)$. 
In other words,  for every $b> 0$, all but a bounded amount of the absolute charges
live on four points.

Given a nonempty subset $A \subset \mathbf{Z}^d$, define
\begin{equation}
	\tmop{Diam} A := \sup_{x, y \in A} \|x - y\|_1
	:=\sup_{x,y\in A}\sum_{i=1}^d|x_i-y_i|;
\end{equation}
this defines the diameter of $A$.
Our next result describes the behavior of the polymer for large values of $\beta$.

\begin{theorem}[Logarithmic diameter]\label{thm:range}
	For all $\beta \in \mathbf{R}$ and $K >  0$
	there exist $0 \leqslant c \leqslant C \leqslant \infty$ such that
	\begin{equation}
		\mathrm{E} \left[ \mathrm{P}_N^{\beta} 
		\left\{ c \leqslant \frac{\tmop{Diam}
		\{S_i :\ 0\leqslant
		 i < N\}}{\ln N} \leqslant C \right\} \right]  \geqslant  1 - N^{- K} ,
  	\end{equation}
	for all sufficiently large integers $N\geqslant 1$. Moreover:
	\begin{enumerate}
	\item If $d \geqslant 1$ and $\mathrm{E} |q_0 |
		< \infty$, then $c>0$.
	\item If $d\geqslant 2$ and
		\begin{equation}\label{eq:cpt:cond}
			\beta > \min_{\alpha \in (0, 1)} \left[
			\beta_{\alpha} \vee \frac{\ln (2 d) }{\alpha \sqrt{\gamma}\,
			\mathrm{E} |q_0 |} \right],
		\end{equation}
		then $C<\infty$.
	\end{enumerate}
\end{theorem}

Therefore, the polymer is ``compact'' for large values of $\beta$ 
in the sense that its diameter grows only logarithmically with
the number of monomers. 

\begin{remark}
	Note, for example, that when the charges
	have the Rademacher distribution [i.e., $q_0=\pm 1$ with
	probability $1/2$ each], condition \eqref{eq:cpt:cond}
	is stating that $\beta > 34 \ln (2 d)$. 
	\qed
\end{remark}

\begin{remark}
	Our proof applies equally well to the case that
	$\beta$ scales with $N$ (see Theorem \ref{thm:log}).
	And the endresult is that, in order to have a ``bounded
	diameter,'' it suffices that $\beta = b \ln N$ for some $b > 0$.
	\qed
\end{remark}

Although the range of the polymer diverges with $N$
(Theorem \ref{thm:range}), one can show that the expectation of $\|S_N\|$ 
remains bounded for all $\beta$ sufficiently large. We describe this phenomenon
next.

Given $\alpha \in (0\,, 1)$ consider the random variable
\begin{equation}
	R_{\alpha}^N := \left\{ \begin{array}{ll}
		\inf \{0 \leqslant i < N :\,
		S_i \in U\} & \text{on } \mathcal{S}_{\alpha},\\
		N & \text{on } \mathcal{S}_{\alpha}^c,
	\end{array}\label{def:R} \right.
\end{equation}
where $U$ is the unique random square that concentrates most of the charges,
given $\mathcal{S}_{\alpha}$. 
The quantity $R_{\alpha}^N$ is therefore the index of the
first monomer that belongs to the unit square $U$ on $\mathcal{S}_{\alpha}$.
And one can use $R_{\alpha}^N$ in order to obtain
a bound on the distance from $U$ to the origin. The distributional
symmetry of the polymer shows that the last monomer on $U$ 
has the same distribution as $N -1 - R^N_{\alpha}$. Therefore, for any $0<\alpha<1$,
\begin{equation}
	\mathrm{E}  \mathrm{E}_N^{\beta} |S_N |  \leqslant  \sqrt{2} + 2
	\mathrm{E}  \mathrm{E}_N^{\beta} \left( R_{\alpha}^N \right) . 
\end{equation}
We will prove later on  that the distribution of $R^N_{\alpha}$ 
has an exponential tail. Our final result is:
\begin{theorem}[Compactness]\label{thm:ER}
	Suppose $d \geqslant 2$, $\alpha \in (0\,, 1)$,
	$\beta > \beta_{\alpha}$, and
	\begin{equation}\label{cond:rho}
		\rho := 2 d \mathrm{E} \left(
		{\rm e}^{- \beta \alpha \sqrt{\gamma} |q_0 |}\right)<1.
	\end{equation}
	Then,
	\begin{equation}\label{eq:cpt}
		\limsup_{N\to\infty} 
		\mathrm{E}  \mathrm{E}_N^{\beta} \left( R_{\alpha}^N \right) 
		\leqslant  \frac{\rho}{(1 - \rho)^2}.
	\end{equation}
\end{theorem}

Condition \eqref{cond:rho} is frequently easy to check.
For example, when $q_0$ has the Rademacher distribution
[i.e., $\mathrm{P}\{q_0=\pm 1\}=1/2$], $\rho=2d\exp(-\beta\alpha/2)$,
and \eqref{cond:rho} holds if and only if $\beta>2\ln(2d)/\alpha$.
Since $\beta_\alpha=32\ln(2d)/(1-\alpha)$, we find that---in the case
of Rademacher-distributed charges---we have
\begin{equation}
	\beta>34\ln(2d)\qquad\Longrightarrow\qquad
	\sup_{N\geqslant 1}\mathrm{E}\mathrm{E}_N^\beta
	(R^N_{1/17})
	\leqslant  \frac{\rho}{(1 - \rho)^2}  <\infty.
\end{equation}

\subsection{On the annealed measure}

Our proofs can be easily adapted to describe the behavior of the 
{\em annealed measure}, defined by
\begin{equation}
	\widetilde{\mathrm{P}}^{\beta}_N \left( A \right):=
	\frac{1}{\mathrm{E} Z_N (\beta)}
	\mathrm{E} \left[ \mathbf{1}_A 
	\exp\left(\frac{\beta}{N} H_N\right) \right],
\end{equation}
when $\mathrm{E} Z_N (\beta) < \infty$. (The latter
condition holds, for example, when $\beta < 1 /\kappa$ 
and $N$ is sufficiently large). The \emph{annealed free energy} is
\begin{equation}
	\widetilde{\digamma} \left( \beta \right) := \lim_{N\to\infty} 
	\frac{1}{N} \ln \mathrm{E} Z_N (\beta) .
\end{equation}
We can define the region of delocalization for the annealed measure and the
annealed critical point respectively  as follows:
\begin{equation}\begin{split}
	\widetilde{\mathscr{D}} &:=
		\left\{ \beta \in \mathbf{R} :\ \lim_{N\to\infty} \mathrm{E} Z_N
		(\beta) = {\rm e}^{\beta} \right\};\\
	\tilde{\beta}_c &:= \sup \widetilde{\mathscr{D}}.
\end{split}\end{equation}
Our results for the annealed measure are similar in flavor to those
for the quenched measure: 
\begin{enumerate}
\item The set $\widetilde{\mathscr{D}}$ is an interval
	that contains $(- \infty\,, 1 / \kappa)$; it coincides with the localized phase
	in the sense that $\| \widetilde{\mathrm{P}}^{\beta}_N - \mathrm{P} \|_{\tmop{TV}}$
	converges to $0$ as $N\to\infty$ if and only if $\beta \in \widetilde{\mathscr{D}}$.
\item Theorem \ref{thm:fo} continues to remain valid after
	we replace $\beta_c$ by $\tilde{\beta}_c$ and
	$\digamma$ by $\widetilde{\digamma}$, and also add the
	restriction---to the set of $\beta$'s---that 
	$\mathrm{E} Z_N (\beta)$ is finite for all large $N$. 
\item The proof of Lemma \ref{lem:Z:eps} shows that $\widetilde{\mathscr{D}} \subset {\mathscr{D}}$, therefore $\tilde{\beta}_c \leqslant \beta_c$; but we believe that this inequality is not sharp in general. 
\end{enumerate}

It is sometimes possible to compute $\tilde{\beta}_c$;
the following highlights an example.

\begin{proposition}\label{prop:bca}
	If $q_0$ has a standard normal
	distribution, then $\mathrm{E} Z_N (1) = \infty$ for all $N\geqslant 1$.
	Consequently, $\tilde{\beta}_c = 1.$
\end{proposition}

We can adapt many of our localization results
to the annealed case provided that $\mathrm{E} Z_N (\beta)$ is finite
and $\beta$ is large [consider for instance charges
that are bounded random variables]. In those cases, as $\beta
\rightarrow \infty$ the trajectory concentrates on two points, while the
charges at a given parity tend to have a constant sign and an absolute value
close to the essential supremum $\|q_0 \|_{L^\infty(\mathrm{P})}$
of the charge distribution.

\subsection{The influence of a pulling force}\label{sec:pulling}

Our proofs will rely only very little on the
assumption that $\{S_i\}_{i=0}^\infty$ is a simple
symmetric random walk. To illustrate, let
us say a few words about the case where $\{S_i\}_{i=0}^\infty$ has a
bias that corresponds to the action of a pulling force.

For every $\lambda \in \mathbf{R}^d$ let us
define a probability measure $\mathrm{P}_{\lambda}$ by the following
prescription of its Radon--Nikod\'ym derivative with respect to $\mathrm{P}$:
For every integer $k\geqslant 1$,
\begin{equation}
	\frac{\mathd \mathrm{P}_{\lambda}}{\mathd \mathrm{P}}
	:= \frac{\exp (\lambda \cdot S_k)}{\mathrm{E}
	\exp (\lambda \cdot S_k)}\qquad
	\text{on $\mathcal{F}_k$},
\end{equation}
where $\mathcal{F}_k$ denotes 
the sigma-algebra generated by all of the charges $\{q_i\}_{i=0}^\infty$
as well as the $k$ initial values $\{S_i\}_{i=0}^k$ of the random walk.

Under the measure $\mathrm{P}_\lambda$ 
the distribution of the charges $q$ remains the same
as that under $\mathrm{P}$, but $S$ becomes
a biased, in particular transient, random walk with 
the following transition probabilities: For every
basis vector $e\in\mathbf{Z}^d$,
\begin{equation}
	\mathrm{P}_{\lambda} \{ S_{k + 1} - S_k = e\}
	= \frac{\exp (\lambda \cdot
	e)}{\mathrm{E} (\exp (\lambda \cdot S_1))}.
\end{equation}

As we did before, in the unforced setting, we consider the measures
\begin{equation}
	\mathrm{P}_N^{\beta, \lambda} (A) := \frac{1}{Z_N (\beta, \lambda)}
	\mathrm{E}_{\lambda} \left[ \left. \mathbf{1}_A \exp \left( \frac{\beta}{N}
	H_N \right)\, \right|\, q_0, q_1, \ldots, q_{N - 1} \right],
\end{equation}
where $Z_N(\beta,\lambda)$ is the partition function,
\begin{equation}
	Z_N (\beta, \lambda) := \mathrm{E}_{\lambda} \left[ \left. \exp \left(
	\frac{\beta}{N} H_N \right)\, \right|\, q_0, q_1, \ldots, q_{N - 1} \right].
\end{equation}
Then we proceed to define the ``$\lambda$--analogues'' of
the quantities of interest. Namely:
\begin{equation}\begin{split}
	\mathscr{D}_{\lambda} & := \left\{ \beta \in \mathbf{R} :\ Z_N (\beta,
		\lambda)  \overset{\mathrm{P}}{\longrightarrow} 
		{\rm e}^{\beta} \text{ as $N\to\infty$}\right\};\\
	\beta_c (\lambda) &:= \sup \mathscr{D}_{\lambda};\qquad\text{and}\\
	\digamma_{\lambda} (\beta) & := \lim_{N\to\infty} 
		\frac{1}{N} \ln Z_N \left( \beta,
		\lambda \right).
\end{split}\end{equation}
Of course, we can write $\mathrm{P}_N^{\beta,\lambda}(A)$ as
follows as well:
\begin{equation}
	\mathrm{P}_N^{\beta,\lambda}(A) =
	\frac{\mathrm{E} \left[ \left. \mathbf{1}_A \exp \left(
	\dfrac{\beta}{N} H_N + \lambda \cdot S_{N - 1} \right)\,
	\right|\, q_0, q_1,
	\ldots, q_{N - 1} \right]}{Z_N (\beta, \lambda) \mathrm{E} (\exp (\lambda
	\cdot S_1))^{N - 1}}. 
\end{equation}
The quantity $\lambda \cdot S_{N - 1}$ is responsible for the different
behavior of $\mathrm{P}_N^{\beta, \lambda}$ from 
$\mathrm{P}_N^\beta$, and corresponds to the potential energy of a
pulling force $\lambda$. 

Define 
\begin{equation}
	I_\lambda(\varepsilon) := \lim_{N\to\infty}
	\frac1N \ln\mathrm{P}_\lambda\left\{ L_N^0>\varepsilon N
	\right\}
	\qquad\text{for all $\varepsilon\in(0\,,1/2)$}.
	\label{eq:Il}
\end{equation}
The proof of Lemma \ref{lem:LT0} below goes through, as no essential
changes are necessary, and ensures that $I_\lambda:(0\,,1/2)
\to(0\,,\infty)$ exists and is continuous.

We will see that Theorem \ref{thm:D}, Proposition
\ref{prop:F}, and Theorem \ref{thm:fo} continue
to remain valid if we respectively replace
$\mathscr{D}$, $\mathrm{P}$, $\mathrm{P}_N^{\beta}$, $\beta_c$, $\digamma$, and $I$
by $\mathscr{D_{\lambda}}$, $\mathrm{P}_\lambda$,$\mathrm{P}_N^{\beta, \lambda}$,
$\beta_c (\lambda)$, $\digamma_{\lambda}$, and $I_{\lambda}$.

We shall also prove that Theorems \ref{thm:square} and \ref{thm:range} continue to
hold, but some of the stated constants need to be changed
because the probability of the trajectory with
maximal energy $H_N$ is no longer $(2 d)^{- N}$. 

Our next result shows that the pulling
force can sometimes trigger the folding/unfolding transition as $\beta_c
(\lambda) \rightarrow \infty$ when $\lambda \rightarrow \infty$.
It also prove that the function
$\lambda \mapsto \beta_c (\lambda)$ is locally Lipschitz continuous.
In order to prepare for that result  let us
observe that the right derivative $\digamma'_\lambda$
of $\digamma_\lambda$ exists everywhere
on $(0\,,\infty)$; this holds by convexity. 

\begin{theorem}\label{thm:pulling}
	If $\mathrm{E} q_0 = 0$, $\tmop{Var} q_0 = 1$, and
	$\kappa < \infty$, then:
	\begin{enumerateroman}
	\item For all $\lambda \in \mathbf{R}^d$,
	\begin{eqnarray}
		\beta_c (\lambda) & \geqslant  &
			\kappa^{-1/2}\cdot\left[ \sqrt{\dfrac{\ln \mathrm{E}
			\exp (\lambda \cdot S_1)}{(\mathrm{E} q_0^+)^2 + (\mathrm{E}
			q_0^-)^2}} \vee \kappa^{-1/2} \right],
			\label{eq:bcg}\\
		\beta_c (\lambda) & \leqslant & \dfrac{2 \ln (2 d) (1
			+\tmmathbf{1}_{\{1\}} (d)) + 4 \left\| \lambda
			\right\|_{\infty}}{(\mathrm{E} q_0^+)^2 + (\mathrm{E} q_0^-)^2}.
			\label{eq:bci}
	\end{eqnarray}
	\item For all $\lambda, \mu \in \mathbf{R}^d$,
	\begin{equation}
		\beta_c (\lambda + \mu)  -  \beta_c (\lambda)
		\leqslant 
		\frac{2\left\| \mu \right\|_{\infty}}{\digamma_{\lambda}' (\beta_c
		(\lambda))},
	\end{equation}
	and $\digamma_{\lambda}' (\beta_c (\lambda))$ satisfies
	\begin{equation}
		\digamma_{\lambda}' (\beta_c (\lambda)) \geqslant  \frac{1}{\beta_c
		(\lambda)} I_{\lambda} \left( \frac{1}{2 \kappa \beta_c
		(\lambda)} \right) .
	\end{equation}
	\end{enumerateroman}
\end{theorem}

\section{Proofs}\label{sec:proofs}

\subsection{Estimates on the partition function}

For every $\varepsilon>0$, we can consider the truncated partition
function
\begin{equation}
	Z_N^{\varepsilon} \left( \beta \right) := \mathrm{E} \left[ \left.
	\tmmathbf{1}_{\{L_N^{\star} \leqslant
	\varepsilon N\}} \exp
	\left( \frac{\beta}{N} H_N \right)\,
	\right|\, q_0, q_1, \ldots, q_{N - 1} \right].
	\label{eq:Zeps}
\end{equation}
		
The following is the main result of this subsection, and  is
essential to our characterization of the delocalized phase.

\begin{proposition}\label{prop:EZ}
	Assume $\mathrm{E} q_0 = 0$ and $\tmop{Var} q_0 = 1$. If $\varepsilon > 0$ and $\beta \in \mathbf{R}$ satisfy either $\beta
	\leqslant 0$ or $2 \kappa \beta \varepsilon < 1$, then
	$\lim_{N\to\infty}\mathrm{E} Z_N^{\varepsilon}  ( \beta) =
	\exp(\beta)$.
\end{proposition}

Note that the above statement implies the convergence $\lim_{N\to\infty}\mathrm{E} Z_N ( \beta) =
\exp(\beta)$ for any $\beta\in\mathbf{R}$ such that $\kappa \beta < 1$, since $L^{\star}_N \leqslant (N + 1) / 2$, and therefore $Z_N (\beta) = Z_N^{(1 / 2) + \delta}	(\beta)$ for all $N\geqslant(2\delta)^{-1}$.

The proof rests on two preparatory lemmas.

\begin{lemma}\label{lem:Z:Jensen}
	Suppose $\mathrm{E} q_0 = 0$ and $\tmop{Var} q_0 = 1$. Let
	$\varepsilon \in (0, 1]$ and $\beta \in \mathbf{R}$ such that either $\beta
 	\leqslant 0$ or $2 \kappa \beta \varepsilon < 1$. Then, for all $\delta > 0$,
	sufficiently small, there exists $C\in(0\,,\infty)$ 
	such that for every $N \geqslant 1$, sufficiently large,
	\begin{equation}
		\mathrm{E} \exp \left( \frac{\beta}{N} \left( q_0 + \cdots + q_{l - 1}
		\right)^2 \right)  \leqslant  \exp \left( \beta \frac{l}{N} + \delta |
		\beta | \frac{l}{N} + C \frac{l^2}{N^2} \right),
	\end{equation}
	uniformly over $l \in \{1\,, \ldots, \lfloor\varepsilon N\rfloor\}$,
\end{lemma}

\begin{lemma}\label{lem:moments}
	Choose and fix $\theta > 0$. Then, as $N \to \infty$,
	\begin{equation}
		\mathrm{E} \left[ \exp \left( \frac{\theta}{N^2} \sum_{x \in
		\mathbf{Z}^d} (L_N^x)^2 \right) \right] \leqslant 1 + \delta_N,
	\end{equation}
	where $\delta_N = O (\ln N / \sqrt{N})$ if $d = 1$, 
	$\delta_N = O ([\ln N]^2  / N)$ if $d = 2$, and 
	$\delta_N = O (\ln N / N)$ if $d \geqslant 3$.
\end{lemma}

Before we prove the two lemmas, let us use them in order to establish
Proposition \ref{prop:EZ}. The lemmas will be proved subsequently.

\begin{proof}[Proof of Proposition \ref{prop:EZ}]
	Let us first note that for all possible realizations of
	$S:=\{S_i\}_{i=0}^\infty$,
	\begin{equation}
		\mathrm{E} \left( \left. H_N\, \right|\, S \right)
		=  \sum_{x \in \mathbf{Z}^d} \mathrm{E} \left[
		\left. (q_1 + \cdots + q_{L_N^x})^2\,\right|\, S \right]
		=  \sum_{x \in \mathbf{Z}^d} L_N^x  =  N.
	\end{equation}
	Therefore, Jensen's inequality implies that 
	$\mathrm{E}[\exp(\beta H_N/N)\,|\,S] \ge
	{\rm e}^\beta$ for all realizations of $S$, whence
	\begin{equation}
		\mathrm{E} Z^{\varepsilon}_N (\beta) 
		\geqslant  {\rm e}^{\beta} \mathrm{P} \left\{
		L^{\star}_N \leqslant \varepsilon N \right\}
		\to {\rm e}^\beta
		\quad\text{as $N\to\infty$};
	\end{equation}
	see Lemma \ref{lem:LTstar} below. This proves half of the assertion
	of the proposition. Next we establish a corresponding upper bound,
	thereby complete the proof.
	
	Thanks to Lemma \ref{lem:Z:Jensen}, for all sufficiently small
	$\delta > 0$ there exists a $C\in(0\,,\infty)$ such that for 
	every $N\geqslant 1$, sufficiently large,
	\begin{equation}\begin{split}
		\mathrm{E} Z^{\varepsilon}_N(\beta) &= \mathrm{E} \left(
			\prod_{x \in \mathbf{Z}^d} \mathrm{E} \left[ \left. \exp \left(
			\frac{\beta}{N} (Q_N^x)^2 \right) \right| S \right] \tmmathbf{1}_{\{L_N^{\star} \leqslant
	\varepsilon N\}}
			\right)\\
		& \leqslant  \mathrm{E} \left[ \exp \left( \beta \sum_{x \in
			\mathbf{Z}^d} \frac{L_N^x}{N} + \delta |
			\beta | \sum_{x \in \mathbf{Z}^d}
			\frac{L_N^x}{N} + C \sum_{x \in \mathbf{Z}^d} \frac{(L_N^x)^2}{N^2}
			\right) \right].
	\end{split}\end{equation}
	Because $\sum_{x \in \mathbf{Z}^d} L_N^x = N$, it follows that
	\begin{equation}
		\mathrm{E} Z^{\varepsilon}_N(\beta)
		\leqslant  \text{\tmtextrm{e}}^{\beta + \delta | \beta |}\,
		\mathrm{E} \exp \left( C \sum_{x \in \mathbf{Z}^d}
		\frac{(L_N^x)^2}{N^2} \right),
	\end{equation}
	and the remainder of the proof follows then from Lemma \ref{lem:moments}.
\end{proof}

Next, we set out to derive Lemmas \ref{lem:Z:Jensen} and 
\ref{lem:moments}, as promised earlier.
\begin{proof}[Proof of Lemma \ref{lem:Z:Jensen}]
	Our goal is to derive a uniform estimate for
	\begin{equation}
		\mathscr{E} := 
		\mathrm{E} \exp \left( \frac{\beta}{N} \left( q_0 + \cdots + q_{l - 1}
		\right)^2 \right).
	\end{equation}
	[This is temporary notation, used specifically for this proof.]
	
	Depending on the sign of $\beta$ we introduce
	the Laplace/Fourier transform
	\begin{equation}
		\Psi (t) := \left\{ \begin{array}{ll}
		\mathrm{E} \exp (t q_0)& \text{if } \beta > 0,\\
		\mathrm{E} \exp (i t q_0) & \text{otherwise}.
		\end{array} \right.
	\end{equation}
	The behavior of $\Psi$ at the origin is given by
	\begin{equation}
		\Psi (t) = \exp \left( \tmop{sgn} (\beta) \dfrac{t^2}{2} + o (t^2)
		\right) \text{ \ \ as \ \ } t \rightarrow 0.
		\label{eq:psi:0}
	\end{equation}
	Furthermore, for all $t \in \mathbf{R}$,
	\begin{equation}
		| \Psi (t) | \leqslant  \left\{ \begin{array}{ll}
		{\rm e}^{\kappa t^2/2}& \text{if } \beta > 0,\\
		1 & \text{otherwise}.
		\end{array} \right.
	\end{equation}
	
	Let $\xi$ be independent of $\{q_i\}_{i=0}^\infty$,
	and have a standard normal distribution. Then,
	\begin{equation}\begin{split}
		\mathscr{E} = \mathrm{E} \exp \left( \sqrt{\frac{2 \beta}{N}} 
			\left( q_0 + \cdots + q_{l - 1} \right) \xi \right)
			&= \mathrm{E} \left( \Psi \left( \sqrt{\frac{2| \beta |}{N}} \xi
			\right)^l \right)\\
		&\leqslant  \mathrm{E} \left( \left| \Psi \left( \sqrt{\frac{2| \beta
			|}{N}} \xi \right) \right|^l \right).
			\label{eq:Eexpl}
	\end{split}\end{equation}
	According to \eqref{eq:psi:0}, there exists some $A (\delta) > 0$ such that
	\begin{equation}
		| \Psi (t) | \leqslant  \exp \left( ( \tmop{sgn} (\beta) + \delta
		) \dfrac{t^2}{2} \right) \quad \text{ when } |t| \leqslant A (\delta).
	\end{equation}
	Because $\mathrm{E} \exp( a \xi^2	) = (1 - 2 a)^{-1/2}$
	for every $a<1/2$, \eqref{eq:Eexpl} implies that
	$\mathscr{E}$ is bounded above by
	\begin{equation}
		\left[1 - 2 (\beta + \delta |
		\beta |) \dfrac{l}{N}\right]^{-1/2} +
		\mathrm{E}\left[
		\exp \left( \varepsilon \kappa \beta^+
		\xi^2 \right) \tmmathbf{1}_{\left\{| \xi | > 
		A (\delta)\sqrt{N/(2\beta)} \right\}}\right].
	\end{equation}
	A Taylor expansion of the logarithm shows that if 
	$\alpha < 1 / 2$ then there exists $C \in(0\,,\infty)$ such
	that $- \frac{1}{2} \ln \left( 1 - 2 \alpha x \right) \leqslant  \alpha x +
	Cx^2$ for all $x \in [0, 1]$.
	Consequently if $\delta > 0$ is sufficiently small, then
	$\ln \mathscr{E}$ is bounded above by
	\begin{align}
		&(\beta + \delta | \beta |) \frac{l}{N} + C
			\frac{l^2}{N^2}\\\nonumber
		&\quad + \ln \left( 1 + \sqrt{1 - 2 (\beta + \delta | \beta |) \frac{l}{N}}
			\mathrm{E} \left[\exp \left( \varepsilon \kappa \beta^+ \xi^2 \right)
			\tmmathbf{1}_{\left\{
			| \xi | > A (\delta)\sqrt{N/(2\beta)}
			\right\}}\right] \right),
	\end{align}
	and the logarithm is at most
	\begin{equation}
		\sqrt{1 + 2| \beta |}\,\mathrm{E}\left[
		\exp \left( \varepsilon \kappa \beta^+
		\xi^2 \right) \tmmathbf{1}_{\left\{
		| \xi | > A (\delta)\sqrt{N/(2\beta)} \right\}}\right].
	\end{equation}
	By the Cauchy--Schwarz inequality, the latter expectation
	vanishes exponentially fast as $N\to\infty$, because
	$\varepsilon \kappa \beta^+ < 1 / 2$; in particular,
	it is uniformly smaller than $l^2 / N^2$ for all
	sufficiently large values of $N$. The lemma follows.
\end{proof}

\begin{proof}[Proof of Lemma \ref{lem:moments}]
	Because $\sum_{x \in \mathbf{Z}^d} (L_N^x)^2
	\leqslant NL_N^{\star}$, it remains to bound $\mathrm{E} [\exp (\theta
	L_N^{\star} / N)]$. First of all, we note that for all $k \geqslant 0$ and
	$N \geqslant 1$,
	\begin{equation}\begin{split}
		\mathrm{E} \left[ (L_N^0)^k \right] &= \underset{0 \leqslant i_1,
			\ldots, i_k < N}{\sum \cdots \sum} \mathrm{P} \left\{ S_{i_1} = \cdots =
			S_{i_k} = 0 \right\} \nonumber\\
		& \leqslant  k! \underset{0 \leqslant i_1 \leqslant \ldots \leqslant i_k
			< N}{\sum \cdots \sum} \mathrm{P} \left\{ S_{i_1} = \cdots = S_{i_k} = 0
			\right\} \nonumber\\
		& \leqslant  k! \left( \mathrm{E} L_N^0 \right)^k . 
	\end{split}\end{equation}
	Consequently,
	\begin{equation}
		\mathrm{E} \left[ \exp \left( \frac{L_N^0}{2 \mathrm{E} L_N^0} \right)
		\right] = \sum_{k = 0}^{\infty} \frac{1}{k!} \mathrm{E} \left[
		\left( \frac{L_N^0}{2 \mathrm{E} L_N^0} \right)^k \right] \leqslant 2. 
		\label{eq:exponentialbound}
	\end{equation}
	Therefore, Chebyshev's inequality, \eqref{eq:exponentialbound},
	and \eqref{eq:SM} together imply 
	that for all $N \geqslant 1$ and  $y > 0$,
	\begin{equation}
		\mathrm{P} \left\{ L_N^{\star} \geqslant yN \right\} \leqslant  2 (2
		N)^d \exp \left( - \frac{yN}{2 \mathrm{E} L_N^0} \right) . 
	\end{equation}
	We will use this bound only if the right-hand side is $\leqslant 1$; i.e.,
	when
	\begin{equation}
		y \geqslant \alpha_N, \hspace{1em} \text{where} \hspace{1em}
		\alpha_N :=
		\frac{2 \mathrm{E} L_N^0 \times \ln [2 (2 N)^d]}{N} .
	\end{equation}
	Else, we use the trivial bound $\mathrm{P} \{L_N^{\star} \geqslant yN\}
	\leqslant 1$. In this way, we find that
	\begin{equation}\begin{split}
		&\int_0^{\infty} \mathrm{P} \left\{ L_N^{\star} \geqslant yN \right\}
			\text{\tmtextrm{e}}^{\theta y} \hspace{0.25em} \text{\tmtextrm{d}} y\\
		&\hskip1.3in\leqslant  \alpha_N + O (N^d) \times
			\int_{\alpha_N}^{\infty} \exp
			\left\{ \theta y - \frac{yN}{2 \mathrm{E} L_N^0} \right\}
			\hspace{0.25em}
			\text{\tmtextrm{d}} y.  \label{eq:int:bd1}
	\end{split}\end{equation}
	Since $\mathrm{E} L_N^0 = \sum_{i = 1}^N \mathrm{P} \{S_i = 0\}$, the
	local-limit theorem [and excursion theory, when $d \geqslant 3$] together
	show that
	\begin{equation}
		\mathrm{E} L_N^0 = (1 + o (1)) \times \left\{ \begin{array}{ll}
		\sqrt{N / \pi} & \text{if } d = 1,\\
		(2 \pi)^{- 1} \ln N& \text{if } d = 2,\\
		1 / \rho (d) & \text{if } d \geqslant 3,
		\end{array} \right.
	\end{equation}
		where $\rho (d) := \mathrm{P} \{\inf_{k \geqslant 1} \|S_k \|> 0\} \in (0
	\hspace{0.25em}, 1)$ for $d \geqslant 3$. It follows readily from this and
	(\ref{eq:int:bd1}) that
	\begin{equation}
		\alpha_N = (1 + o (1)) \times \left\{ \begin{array}{ll}
		2 \ln N / \sqrt{\pi N} & \text{if } d = 1,\\
		2 (\ln N)^2 / (\pi N) & \text{if } d = 2,\\
		2 d \ln N / (\rho (d) N) & \text{if } d \geqslant 3.
		\end{array} \right.
	\end{equation}
	Moreover,
	\begin{equation}\begin{split}
		&\int_0^{\infty} \mathrm{P} \left\{ L_N^{\star} \geqslant yN \right\}
			\text{\tmtextrm{e}}^{\theta y} \hspace{0.25em} \text{\tmtextrm{d}} y\\
		&\hskip1.3in
			\leqslant  \alpha_N + O (N^d) \times \int_{\alpha_N}^{\infty} \exp
			\left\{ \theta y - \frac{yN}{2 \mathrm{E} L_N^0} \right\} \hspace{0.25em}
			\text{\tmtextrm{d}} y, 
	\end{split}\end{equation}
  and direct computations show that the preceding is $O (\ln N / \sqrt{N})$ if
  $d = 1$, $O ([\ln N]^2 / N)$ if $d = 2$, and $O (\ln N / N)$ if $d \geqslant
  3$. Integration by parts then shows that
  \begin{equation}
    \mathrm{E} \left[ \exp \left( \frac{\theta L_N^{\star}}{N} \right) \right]
    = 1 + \int_0^{\infty} \mathrm{P} \left\{ L_N^{\star} \geqslant yN \right\}
    \text{\tmtextrm{e}}^{\theta y} \hspace{0.25em} \text{\tmtextrm{d}} y.
  \end{equation}
  Therefore, the lemma follows from the bound $\sum_{x \in \mathbf{Z}^d}
  (L_N^x)^2 \leqslant NL_N^{\star}$.
\end{proof}

\subsection{The delocalized phase}

Before we give the proof of Theorem \ref{thm:D}, we state and prove an easy
consequence of Proposition~\ref{prop:EZ}:
\begin{lemma}\label{lem:Z:eps}
	Assume $\mathrm{E} q_0 = 0$, $\tmop{Var} q_0 = 1$ and $\kappa <\infty$. Let $\varepsilon > 0$
	and $\beta \in \mathbf{R}$ such that either $\beta \leqslant 0$ or $2 \kappa
  	\beta \varepsilon < 1$. Then
	\begin{equation}
		Z_N^{\varepsilon} (\beta)
		\overset{\mathrm{P}}{\longrightarrow} {\rm e}^{\beta}
		\qquad\text{as $N\to\infty$}.
	\end{equation}
\end{lemma}

\begin{proof}
First, we prove that, when $\beta \leqslant 0$ or $4 \kappa \beta \varepsilon < 1$,
\begin{equation}
		Z_N^{\varepsilon} (\beta)
		\overset{L^2 ( \mathrm{P})}{\longrightarrow} {\rm e}^{\beta}
		\qquad\text{as $N\to\infty$} \label{ZL2}.
	\end{equation}
	Because $(Z_N^{\varepsilon} (\beta))^2 \leqslant
	Z_N^{\varepsilon} (2 \beta)$ [Jensen's inequality],
	\begin{equation}
		\mathrm{E}\left(\left|
		Z_N^{\varepsilon} (\beta) - {\rm e}^{\beta}
		\right|^2\right) \leqslant
		\mathrm{E} Z_N^{\varepsilon} (2 \beta) +
		{\rm e}^{2 \beta} - 2 {\rm e}^{\beta}
		\mathrm{E} Z_N^{\varepsilon} (\beta).
	\end{equation}
	The latter quantity goes to 0 as $N\to\infty$, thanks to
	Proposition \ref{prop:EZ}, and this proves (\ref{ZL2}). Now we conclude the proof of the Lemma and assume $\beta \leqslant 0$ or $2 \kappa \beta \varepsilon < 1$. The variable $Z_N^{\varepsilon} (\beta)-Z_N^{\varepsilon/2} (\beta)$ is non-negative and its expectation goes to $0$ as $N\to\infty$, cf.~Proposition~\ref{prop:EZ}. Therefore it converges to $0$ in probability. By (\ref{ZL2}) we know already that $Z_N^{\varepsilon/2} (\beta)\to {\rm e}^{\beta}$ in probability as $N\to\infty$. The conclusion follows.
\end{proof}

\begin{proof}[Proof of Theorem \ref{thm:D}]
	Let us first prove that $(- \infty\,, 1 / \kappa)
	\subseteq\mathscr{D}$. We choose and fix $\beta \in(-\infty\,,1 / \kappa)$. There is $\delta > 0$ such that $2 \kappa\beta(\frac12 + \delta) < 1$. We have seen already that $Z_N (\beta) = Z_N^{(1 / 2) + \delta}
	(\beta)$ for all $N\geqslant(2\delta)^{-1}$, therefore $\beta \in \mathscr{D}$ is a consequence of Lemma~\ref{lem:Z:eps}.

	Next we prove that $\mathscr{D}$ is an interval.  Thanks to the topology of
	$\mathbf{R}$, it suffices to show
	that $\mathscr{D}\cap(0,\infty)$ is connected.
	
	Let us choose and fix $\beta_1, \beta_2 \in
	\mathscr{D}$ such that $0 < \beta_1 < \beta_2$.
	For all $\beta \in (\beta_1\,,\beta_2)$ and $\gamma>1$, 
	$(Z_N (\beta))^{\gamma}
	\leqslant Z_N (\gamma \beta)$, thanks to the conditional
	Jensen inequality. It follows that
	$Z_N (\beta_1)^{\beta / \beta_1} \leqslant Z_N (\beta) \leqslant Z_N
	(\beta_2)^{\beta_2 / \beta}$. We can pass
	to the limit $[N\to\infty]$ to deduce
	that $\beta \in \mathscr{D}$. This implies the connectivity 
	of $\mathscr{D}$, and completes the proof of part 1.

	Assertion 2 of the theorem holds because
	\begin{equation}
		\mathrm{P}_N^{\beta} \left\{ L^{\star}_N \leqslant \varepsilon N
		\right\} =
		\frac{Z^{\varepsilon}_N (\beta)}{Z_N (\beta)},
	\end{equation}
	and $Z^{\varepsilon}_N (\beta) \rightarrow {\rm e}^{\beta}$ in probability for
	all sufficiently small $\varepsilon > 0$ [Lemma \ref{lem:Z:eps}]. 
  
	Finally we demonstrate part 3. 
	Assume first $\beta \nin \mathscr{D}$. For $N$ fixed, the total variation is at least $\mathrm{P}_N^\beta\{L^{\star}_N \leqslant \varepsilon N\}-\mathrm{P}\{L^{\star}_N \leqslant \varepsilon N\}$, which does not converge to $0$ in probability as $N\to\infty$ according to assertion 2 and to Lemma~\ref{lem:LTstar}.
	
	Now we consider  $\beta \in \mathscr{D}$ and 
	$\varepsilon > 0$ such that $4 \kappa \beta^+ \varepsilon < 1$,
	and consider some event $A$ that might depend on all $\{S_i\}_{i=0}^\infty$ and $\{q_i\}_{i=0}^\infty$. We have
	\begin{equation}
		\left| \mathrm{P}_N^\beta\left(A\right)-\mathrm{P}\left(A\left|q_0,\ldots,q_{N-1}\right.\right)
		\right| \leqslant 
		 d_1 + d_2
	\end{equation}
	where
	\begin{equation}\begin{split}
	d_1&:=\left| \mathrm{P}_N^\beta\left(A\cap\{L^{\star}_N \leqslant \varepsilon N\}\right)-\mathrm{P}\left(\left. A\cap\{L^{\star}_N \leqslant \varepsilon N\}\right|q_0,\ldots,q_{N-1} \right)
		\right|,\text{ and}\\
	d_2&:=\mathrm{P}_N^\beta\left(\{L^{\star}_N > \varepsilon N\}\right) + \mathrm{P}\left(\left. \{L^{\star}_N > \varepsilon N\}\right|q_0,\ldots,q_{N-1} \right).
 \end{split}
	\end{equation}
	According to assertion 2 and to Lemma~\ref{lem:LTstar}, $d_2\to 0$ in probability as $N\to\infty$. So it suffices to prove that $d_1\to0$ in probability as $N\to\infty$, uniformly in $A$.
	It follows from the definition of $\mathrm{P}_N^\beta$ that
	\begin{equation}\begin{split}
	d_1&\leqslant \mathrm{E}\left[ \left. \left|	
	\frac{\exp \left( \beta H_N/N \right) }{Z_N(\beta)}
-	1 \right| \mathbf{1}_{A\cap\{L^{\star}_N \leqslant \varepsilon N\}}
	  \right| q_0,
	q_1, \ldots, q_{N - 1} \right] \\
	&\leqslant
	{ \mathrm{E}\left[ \left. \left|	
	\frac{\exp \left( \beta H_N/N \right) }{Z_N(\beta)}
-	1 \right|^2 \mathbf{1}_{\{L^{\star}_N \leqslant \varepsilon N\}}
	  \right| q_0,
	q_1, \ldots, q_{N - 1} \right] }^{1/2}\\
	&=  \left[\frac{Z^\varepsilon_N(2\beta)}{Z_N(\beta)^2}
-2\frac{Z^\varepsilon_N(\beta)}{Z_N(\beta)}
+\mathrm{P}\left(\left. \{L^{\star}_N \leqslant \varepsilon N\}\right|q_0,\ldots,q_{N-1} \right)\right]^{1/2}.
	\end{split}
	\end{equation}
And the latter quantity, which does not depend on $A$, goes to zero in probability as $N\to\infty$;
	see Lemma \ref{lem:Z:eps} and Lemma \ref{lem:LTstar}.
\end{proof}

Finally we prove the invariance principle of the introduction.

\begin{proof}[Proof of Corollary \ref{cor:BM}]
	Theorem \ref{thm:D} implies that $\mathrm{E}^{\beta}_N
	\left[ \Phi \left( \mathscr{S}_N \right) \right] - \mathrm{E} [\Phi (
	\mathscr{S}_N)]$ converges in probability to zero, as $N\to\infty$. 
	And, according to Donsker's invariance
	principle, $\mathrm{E} [\Phi ( \mathscr{S}_N)] \to \mathrm{E}
	[\Phi ( \mathscr{B})]$. The corollary follows immediately from
	these observations.
\end{proof}

\subsection{The existence of free energy (proof of Proposition \ref{prop:F})}

In this section we show that the normalized energy $H_N / N$ is subadditive,
and then conclude Proposition \ref{prop:F} from that fact.

\begin{lemma}\label{lem:subadd}
	Let $N_1, N_2 \geqslant 1$ and $\tilde{q} :=\{q_{N + i} \}_{i = 0}^{\infty}$,
	$\tilde{S} :=\{S_{N_1 + i} - S_{N_1} \}_{i = 0}^{\infty}$, $\tilde{Q}_N^x :=
	\sum_{i = 0}^{N - 1} \tilde{q}_i \tmmathbf{1}_{\{\tilde{S}_i = x\}}$, and
	$\tilde{H}_N := \sum_{x \in \mathbf{Z}^d} ( \tilde{Q}_N^x)^2$. Then,
	\begin{equation}\label{eq:subadd}
		\frac{H_{N_1 + N_2}}{N_1 + N_2} \leqslant  \frac{H_{N_1}}{N_1} +
		\frac{\tilde{H}_{N_2}}{N_2} \qquad\text{a.s.\ $[\mathrm{P}]$.}
	\end{equation}
	Furthermore, $H_{N_1}$ and $H_{N_2}$ are conditionally
	independent, given $\{q_i\}_{i=0}^\infty$, and
	the conditional distribution of $\tilde{H}_{N_2}$ is the same as the
	conditional distribution of $H_{N_2}$ given
	the charges $\tilde{q}$.
\end{lemma}

\begin{proof}
	Clearly,
	\begin{equation}
		Q_{N_1 + N_2}^x = Q_{N_1}^x + \tilde{Q}_{N_2}^{x + S_{N_1}}
		\qquad\text{for every $x\in\mathbf{Z}^d$}.
	\end{equation}
	Therefore, the convexity of $h(x):= x^2$ implies that
	\begin{equation}
		\frac{1}{N_1 + N_2} \left( Q_{N_1 + N_2}^x \right)^2 \leqslant 
		\frac{1}{N_1}  \left( Q_{N_1}^x \right)^2 + \frac{1}{N_2}  \left(
		\tilde{Q}_{N_2}^{x + S_{N_1}} \right)^2 .
	\end{equation}
	We can sum the preceding over all $x \in \mathbf{Z}^d$ to deduce
	\eqref{eq:subadd}. In addition, the conditional
	distribution of $\tilde{H}_{N_2}$, given the charges $\tilde{q}$, 
	depends only
	on the distribution of $\tilde{S}$, which is the law of a simple random walk.
\end{proof}

\begin{proof}[Proof of Proposition \ref{prop:F}]
	Let
	\begin{equation}\label{eq:FN}
		\digamma^q_N (\beta) := \frac{1}{N} \ln Z_N \left( \beta \right)
		:= \left. \frac{1}{N} \ln \mathrm{E} \left[ \exp \left( \beta
		\frac{H_N}{N} \right) \,\right|\, q_0, q_1, \ldots, q_{N - 1}
		\right]
	\end{equation}
	denote the free energy corresponding to a finite and fixed $N\geqslant 1$
	and to a given realization of the charges $q:=\{q_i\}_{i=0}^\infty$. 
	
	By the conditional Jensen's inequality,
	\begin{equation}
		\liminf_{N\to\infty} 
		\mathrm{E} \left[\digamma_N^q \left( \beta \right) \right]\geqslant 
		\beta \lim_N \mathrm{E} \left(  \frac{H_N}{N^2}  \right) = 0, \label{King1}
	\end{equation}
  	since as $N\to\infty$,
	\begin{equation}
		\mathrm{E} H_N = N \tmop{Var} (q_0) + ( \mathrm{E} q_0)^2 \mathrm{E}
		\sum_{x \in \mathbf{Z}^d} (L_N^x)^2 =o(N^2);
	\end{equation}
	see Lemma \ref{lem:moments}. This proves that if $\digamma(\beta)$ exists
	[as the proposition asserts]
	and is nonrandom, then certainly $\digamma(\beta)\geqslant 0$.
	
	Now we prove convergence. 
	
	According to Lemma \ref{lem:subadd}, for every fixed $N_1,
	N_2 \geqslant 1$, we can bound $\digamma^q_{N_1 + N_2} (\beta)$
	from above by
	\begin{align}\nonumber
		 & \left. \frac{1}{N_1 + N_2}
			\ln \mathrm{E} \left[ \exp
			\left( \beta \frac{H_{N_1}}{N_1} \right) \times
			\exp \left( \beta \frac{\tilde{H}_{N_2}}{N_2} \right)
			\,\right|\, q_0, q_1, \ldots,
			q_{N_1 + N_2 - 1} \right]\\
		&\hskip1.5in= \frac{1}{N_1 + N_2} \left( N_1 \digamma^q_{N_1} (\beta) + N_2
			\digamma_{N_2}^{\tilde{q}} (\beta) \right).
	\end{align}
	Because $\digamma_1^q (\beta) = q_0^2$ has a
	finite expectation and because of the minoration (\ref{King1}),
	Kingman's subadditive ergodic theorem \cite{K68,K73}
	tells us that $\digamma^q_N(\beta)$ converges a.s.\ and in $L^1(\mathrm{P})$.
	In particular,
	\begin{equation}
		\digamma (\beta) = \lim_{N\to\infty} 
		\frac{1}{N}  \mathrm{E} \ln Z_N \left( \beta
		\right) .  \label{eq:Fking}
	\end{equation}
	The monotonicity and the convexity of $\beta \mapsto N^{- 1} \ln Z_N
	(\beta)$, and hence of $\digamma$, follow respectively from
	the following relations:
	\begin{equation}\label{eq:F'F''}\begin{split}
		\frac{\mathd}{\mathd \beta} \left(
			\digamma^q_N(\beta)\right) & =
			\frac{Z_N'(\beta)}{NZ_N(\beta)}=
			\mathrm{E}_N^{\beta} \left(  \frac{H_N}{N^2}  \right);\\
		\frac{\mathd^2}{\mathd \beta^2}  \left(
			\digamma^q_N(\beta)\right) 
			&= \frac{Z_N''(\beta)Z_N(\beta)-
			\left[Z_N'(\beta)\right]^2}{N\left[Z_N(\beta)\right]^2}=
			\tmop{Var}_{\mathrm{P}_N^{\beta}}
			\left(  \frac{H_N}{N^{3/2}}  \right);
	\end{split}\end{equation}
	together with the fact that both of these quantities 
	are nonnegative.
\end{proof}

\subsection{The first-order phase transition (proof of Theorem \ref{thm:fo})}

Our proof of Theorem \ref{thm:fo} requires three preliminary Lemmas.

\begin{lemma}\label{lem:jump}
	For all $\beta> 0$ and $\varepsilon,\eta>0$,
	\begin{equation}
		\mathrm{P}_N^{\beta} \left\{ \varepsilon<
		\frac{L^{\star}_N}{N} \leqslant 
		\frac{1-\eta}{2 \kappa \beta}\right\}
		\overset{\mathrm{P}}{\longrightarrow}  0
		\qquad\text{as $N\to\infty$}.
	\end{equation}
\end{lemma}

\begin{proof} We assume of course that $\varepsilon<(1-\eta)/(2 \kappa \beta)$.
	Because $Z_N (\beta) \geqslant 1$,
	\begin{equation}\begin{split}
		\mathrm{E}  \left[\mathrm{P}_N^{\beta} \left\{ 
			\varepsilon < \frac{L^{\star}_N}{N} \leqslant
			\frac{1-\eta}{2 \kappa \beta}  \right\}
			\right]
			&= \mathrm{E}\left[\frac{Z_N^{(1-\eta) / (2 \kappa \beta)} (\beta) -
			Z_N^{\varepsilon} (\beta)}{Z_N (\beta)}\right]\\
		&\leqslant  \mathrm{E} \left[ Z_N^{(1-\eta) / (2 \kappa \beta)}
			(\beta) - Z_N^{\varepsilon} (\beta) \right].
	\end{split}\end{equation}
	This proves the lemma because according to Proposition \ref{prop:EZ}
	the preceding converges to zero as $N\to\infty$.
\end{proof}

\begin{lemma}\label{lem:densF}
	If $\mathrm{E} (q_0^2) = 1$, then for all $\varepsilon,\beta>0$,
	\begin{equation}
		\mathrm{P}_N^{\beta} \left\{
		\frac{L^{\star}_N}{N} \geqslant \frac{\digamma
		(\beta)}{\beta} - \varepsilon \right\}
		\overset{\mathrm{P}}{\longrightarrow}  1
		\qquad\text{as $N\to\infty$}.
	\end{equation}
\end{lemma}

\begin{proof}
	Whenever we have $H_N/N^2 \leqslant -\varepsilon+[\digamma(\beta)/\beta]$,
	then we certainly have
	$\exp(\beta H_N/N)\le\exp(N\digamma(\beta)-\beta
	\varepsilon N).$ Therefore,
	\begin{equation}
		\mathrm{P}^{\beta}_N \left\{ \frac{H_N}{N^2}
		\leqslant \frac{\digamma (\beta)}{\beta} - 
		\varepsilon \right\} \leqslant  \frac{\mathrm{e}^{N
		\digamma( \beta) -\beta\varepsilon N}}{Z_N (\beta)}.
	\end{equation}
	It follows from Proposition \ref{prop:F} that for every $\varepsilon>0$,
	\begin{equation}
		\mathrm{P}^{\beta}_N \left\{ \frac{H_N}{N^2} \geqslant \frac{\digamma
		(\beta)}{\beta} - \varepsilon \right\}
		\overset{\mathrm{P}}{\longrightarrow}  1
		\quad\text{as $N\to\infty$}.
		\label{eq:densF}
	\end{equation}
	Next we prove that the preceding implies the result.
	
	In accord with the Cauchy--Schwarz inequality,
	\begin{equation}\begin{split}
		(Q_N^x)^2  &\leqslant  \left(
			\sum_{i = 1}^N q_i^2 \tmmathbf{1}_{\{S_i =
			x\}} \right) \times L_N^x\qquad\text{for all $x\in\mathbf{Z}^d$}\\
		&\leqslant  \left(
			\sum_{i = 1}^N q_i^2 \tmmathbf{1}_{\{S_i =
			x\}} \right) \times L_N^\star.
	\end{split}\end{equation}
	We sum this inequality over $x \in
	\mathbf{Z}^d$ to find that
	\begin{equation}\label{eq:H:Lstar}
		H_N \leqslant L^{\star}_N \cdot\sum_{i = 1}^N q_i^2.
	\end{equation}
	The lemma follows from \eqref{eq:densF} and the law of large numbers.
\end{proof}

\begin{lemma}\label{lem:LH}
	For every $\varepsilon, \beta > 0$ and $0<\delta < I
	(\varepsilon) / \beta$,
	\begin{equation}
		\mathrm{P}_N^{\beta} \left\{ H_N \leqslant \delta N^2,\,
		L^{\star}_N \geqslant \varepsilon N \right\}
		\overset{\mathrm{P}}{\longrightarrow} 0
		\quad\text{as $N\to\infty$}.
	\end{equation}
\end{lemma}

\begin{proof}
	According to Lemma \ref{lem:LTstar},
	\begin{equation}\begin{split}
		&\limsup_{N\to\infty} \frac{1}{N} \ln \mathrm{E} \left[\left. 
			\mathrm{e}^{\beta H_N/N}\tmmathbf{1}_{\{H_N \leqslant \delta N^2
			,\, L^{\star}_N \geqslant \varepsilon N\}} \,\right|\, q_0\,, \ldots,
			q_{N - 1} \right]\\
		&\hskip3in\leqslant  \beta \delta - I (\varepsilon) < 0,
	\end{split}\end{equation}
	almost surely. Because $Z_N (\beta)\geqslant 1$,
	$\mathrm{P}_N^{\beta}\{ H_N \leqslant \delta N^2,\,
	L^{\star}_N \geqslant \varepsilon N\}$ is a.s.\
	bounded above by the conditional
	expectation in the preceding display.
\end{proof}

\begin{proof}[Proof of Theorem \ref{thm:fo}]
For all $\beta\in{\mathbf{R}}$ we define
	\begin{equation}
		\gamma(\beta) := \lim_{\varepsilon\downarrow 0}
		\limsup_{N\to\infty}
		\mathrm{E}\left[\mathrm{P}_N^{\beta} \left\{ \frac{L^{\star}_N}{N}
		\geqslant \varepsilon \right\}\right].
	\end{equation}
Theorem \ref{thm:D} shows that $\gamma(\beta)>0$ if and only if $\beta\nin \mathscr{D}$.
We will prove that, for all $\beta>0$,
	\begin{equation}
		\lim_{\delta\downarrow0} \frac{\digamma(\beta+\delta)-\digamma(\beta)}{\delta}
		 \geqslant 
		  \frac{\gamma(\beta)}{\beta} I \left( \frac{1}{2 \kappa \beta} \right). \label{eq:lwbFp}
	\end{equation}
Before we address the proof of (\ref{eq:lwbFp}), we explain how it implies (\ref{eq:fo}). For any $\beta>\beta_c$, we have $\gamma(\beta)>0$ [Theorem \ref{thm:D}] and therefore a consequence of (\ref{eq:lwbFp}) is that $\digamma(\beta)>0$, for all $\beta>\beta_c$. Then, from Lemma~\ref{lem:densF} it follows that $\digamma(\beta)>0 \Rightarrow \gamma(\beta)=1$, therefore $\gamma(\beta)=1$ for all $\beta>\beta_c$, and reporting in (\ref{eq:lwbFp}) yields the positive slope of $\digamma$ at the critical point, that is (\ref{eq:fo}).  
Eq.\ \eqref{eq:prop:min} follows from the fact that $\digamma(\beta)>0$ for all $\beta>\beta_c$, together with Lemmas \ref{lem:densF} and \ref{lem:jump}.
	
Now we turn to the proof of (\ref{eq:lwbFp}). We fix $\beta>0$ and $\varepsilon>0$. 
According to Lemma \ref{lem:jump} we have as well
\begin{equation}
		 \limsup_{N\to\infty} \mathrm{E}
		\left[ \mathrm{P}_N^{\beta} \left\{\frac{L^{\star}_N}{N}
		\geqslant \frac{1 - \varepsilon}{2 \kappa \beta} \right\}\right] = \gamma(\beta) .
	\end{equation}
Since  $Z_N (\beta)$ and $Z_N^{\varepsilon} (\beta)$ are nondecreasing
	functions of $\beta$,
	\begin{equation}\begin{split}
		&\inf_{\eta \in [0, \delta]} \mathrm{P}_N^{\beta + \eta} \left\{
			\frac{L^{\star}_N}{N} > \frac{1 - \varepsilon}{2 \kappa \beta} \right\}\\
		&\hskip1in\geqslant 1 - 
			\frac{Z_N^{(1- \varepsilon)/(2 \kappa \beta)} (\beta +
			\delta)}{Z_N (\beta)}\\
		&\hskip1in\geqslant  
		\mathrm{P}_N^{\beta} \left\{ \frac{L^{\star}_N}{N} >
		\frac{1- \varepsilon }{2 \kappa
		\beta} \right\} - \varepsilon,
	\end{split}\end{equation}
	almost surely on $\mathcal{T}^\delta_N$ where
		\begin{equation}
		\mathcal{T}^\delta_N  :=  \left\{ 
		\frac{Z_N^{(1- \varepsilon)/(2 \kappa \beta)}(\beta + \delta)-Z_N^{(1- \varepsilon)/(2 \kappa\beta)}(\beta)}{Z_N(\beta)} 
		\leqslant \varepsilon 
		\right\}.
	\end{equation}
	According to Lemma \ref{lem:Z:eps}, for all $\delta>0$ small enough, $Z_N^{(1- \varepsilon)/(2 \kappa\beta)}(\beta) \to {\rm e}^\beta$ while
	$Z_N^{(1- \varepsilon)/(2 \kappa\beta)}(\beta+\delta) \to {\rm e}^{\beta+\delta}$ in probability, as $N\to\infty$. Consequently $\mathrm{P}(\mathcal{T}^\delta_N)\to1$ and
	\begin{equation}
		  \limsup_{N\to\infty} \inf_{\eta \in [0, \delta]} \mathrm{E}\left[
		\mathrm{P}_N^{\beta + \eta}
		\left\{\frac{L^{\star}_N}{N}
		\geqslant \frac{1 - \varepsilon}{2 \kappa \beta} \right\}
		\right] \geqslant \gamma(\beta)-\varepsilon
	\end{equation}
	for all $\delta>0$ sufficiently small.	
	In view of Lemma \ref{lem:LH}, this yields also
	\begin{equation}
		  \limsup_{N\to\infty} 
		\inf_{\eta \in [0, \delta]} \mathrm{E}\left[
		\mathrm{P}_N^{\beta + \eta} \left\{
		\frac{H_N}{N^2} \geqslant \frac1\beta I \left( \frac{1- \varepsilon}{2 \kappa \beta} 
		\right) \right\}\right]
		\geqslant \gamma(\beta)-\varepsilon.
	\end{equation}
	Consequently, we can integrate \eqref{eq:F'F''}
	over all $\eta\in(\beta\,,\beta+\delta)$ to see that
	\begin{equation}
		 \limsup_{N\to\infty} \left[ \mathrm{E}\left[\digamma_N (\beta + \delta) \right]- 
		\mathrm{E}\left[\digamma_N (\beta)\right] \right]
		\geqslant \delta \frac{\gamma(\beta)-\varepsilon}\beta
		I\left( \frac{1 - \varepsilon}{2 \kappa
		\beta} \right)
	\end{equation}
	and letting $\varepsilon\to0$  we conclude the proof of (\ref{eq:lwbFp}).	
\end{proof}

\subsection{Energy and the distance to optimality: The four points
(Proofs of Proposition \ref{prop:maxH:bc} and Theorem 
\ref{thm:square})}

The aim of this subsection is to prove Proposition \ref{prop:maxH:bc} and
Theorem \ref{thm:square}. We consider henceforth the following related 
problem: What is the maximum value of
$H_N$ given $q_0, \ldots, q_{N - 1}$, where the maximum is taken over
all possible random walk paths. 

Let us introduce some notation. We say
that $x \in \mathbf{Z}^d$ is \emph{odd} (resp.\ 
\emph{even}) when the sum of its coordinates
is odd (resp.\ even). Given $N \geqslant 1$, $\varepsilon \in \{-\,,+\}$,
and $p \in \{\tmop{odd}\,, \tmop{even}\}$ we define
\begin{equation}
	Q_{\varepsilon}^p := \sum_{\substack{0\leqslant
	i < N :\\i \equiv p}} q_i^{\varepsilon},
	\label{eq:Qep}
\end{equation}
where ``$i \equiv p$'' means that ``$i$ has parity $p$.'' The quantity
$Q_{\varepsilon}^p$ is the total value of charges of sign $\varepsilon$
available at positions of parity $p$. 

Given a realization of $(q\,, S)$ we
define $x_{\varepsilon}^p$ as any one of the points of $\mathbf{Z}^d$ with parity
$p$ such that $\varepsilon Q_N^x$ is maximal (since positions with no charge
exist we always have $\varepsilon Q_N^{x_{\varepsilon}^p} \geqslant 0$).
It is not hard to see that we can ensure that $x_\varepsilon^p$ is always
a random variable [measurable with respect to the sigma-algebra
generated by $(q\,,S)$].

Let us also observe that if there exists a point $x$ of parity 
$p$ such that $\varepsilon Q_N^x > Q_{\varepsilon}^p / 2$, 
then there is a unique choice for $x_{\varepsilon}^p$, namely
$x_\varepsilon^p=x$. 

We may think of
\begin{equation}
	D_N := \mathop{\sum\sum}\limits_{\substack{
	\varepsilon\in\{-,+\}\\p\in\{\text{odd},\text{even}\}}} Q_{\varepsilon}^p \left(
	Q_{\varepsilon}^p - \varepsilon Q_N^{x_{\varepsilon}^p}\right)
	\label{eq:DN}
\end{equation}
as the {\tmem{charge distance to optimality}}. 
Clearly, $D_N\geqslant 0$.

\begin{lemma}\label{lem:HD}{\tmdummy}
	The following are valid for all $N\geqslant 1$:  
	\begin{enumerateroman}
		\item For every $d \geqslant 1$,
		\begin{equation}
			H_N  \leqslant  
			\mathop{\sum\sum}\limits_{\substack{
			\varepsilon\in\{-,+\}\\p\in\{\text{\rm odd},\text{\rm
			even}\}}} 
			(Q_{\varepsilon}^p)^2 - D_N . 
		\label{eq:HDN}
		\end{equation}
	\item For every $d \geqslant 2$,
		\begin{equation}
			\max_S H_N (S)= 
			\mathop{\sum\sum}\limits_{\substack{
			\varepsilon\in\{-,+\}\\p\in\{\text{\rm odd},\text{\rm even}\}}} 
			(Q_{\varepsilon}^p)^2,
		\label{eq:Hmax:srw}
		\end{equation}
		where ``$\max_{_S}$'' refers to the maximum over all possible
		random walk paths.
	\end{enumerateroman}
\end{lemma}

\begin{proof}
	In order to prove part 1 we first decompose,
	and then estimate, the energy as follows:
	\begin{align}
			H_N &= \mathop{\sum\sum}\limits_{\substack{
				\varepsilon\in\{-,+\}\\p\in\{\text{\rm odd},\text{\rm even}\}}} 
				 \sum_{\substack{x\in\mathbf{Z}^d:\, 
				 x \equiv p,\\\varepsilon Q_N^x
				> 0}} \left( Q_N^x \right)^2\\
			& \leqslant\mathop{\sum\sum}\limits_{\substack{
				\varepsilon\in\{-,+\}\\p\in\{\text{\rm odd},\text{\rm even}\}}}
				\max_{\substack{x\in\mathbf{Z}^d:\, x \equiv p,\\
				\varepsilon Q_N^x > 0}} \left( \varepsilon Q_N^x \right) 
				\times \sum_{\substack{x\in\mathbf{Z}^d:\, x \equiv p,\\
				\varepsilon Q_N^x > 0}} \varepsilon Q_N^x\\
			& \leqslant  \mathop{\sum\sum}\limits_{\substack{
				\varepsilon\in\{-,+\}\\p\in\{\text{\rm odd},\text{\rm even}\}}}
				\varepsilon Q_N^{x_{\varepsilon}^p}
				\times Q_{\varepsilon}^p.
		\end{align}
	We express the latter in terms of $D_N$ to complete the
	proof of part 1.
	
	Next we demonstrate part 2.
	
	Thanks to part 1 of the lemma,
	\begin{equation}
		\max_S H_N \leqslant\mathop{\sum\sum}\limits_{\substack{
		\varepsilon\in\{-,+\}\\p\in\{\text{\rm odd},\text{\rm even}\}}}
   		 (Q_{\varepsilon}^p)^2
		 \quad\text{for all $d\geqslant 1$}.
	\end{equation}
	Now we assume $d \geqslant 2$ and describe an ``optimal
	trajectory'' in order to establish the second part of the lemma.
	
	In order to be concrete, we will consider the case that
	$q_0 \geqslant 0$; the case that $q_0<0$ can be considered
	similarly. Define
	\begin{alignat}{3}
		\sigma_+^{\tmop{even}} &:=& (0, 0, 0,\ldots),\qquad
		\sigma_-^{\tmop{even}} &:=& (1, 1, 0, \ldots),\\
		\sigma_+^{\tmop{odd}} &:=& (0, 1, 0, \ldots),\qquad
		\sigma_-^{\tmop{odd}} &:=& (1, 0, 0,\ldots).
	\end{alignat}
	[When $q_0<0$, we exchange the roles of
	$\sigma_+^{\tmop{even}}$ and $\sigma_-^{\tmop{even}}$ 
	in the following argument.] Now let us consider the following 
	possible random walk trajectory:
	\begin{equation}
		S_i = \sigma_{\tmop{sgn} (q_i)}^{\tmop{parity} (i)} 
		\quad\text{for $i\geqslant 0$}
	\end{equation}
	A direct inspection shows that: (i)
	$S$ is a realization of the simple random walk; and (ii) this realization
	of the random walk path achieves the maximum energy $\max_{_S}H_N$.
	[In particular, for this realization of the random walk we have
	$x^p_\varepsilon=\sigma^p_\varepsilon$.]
\end{proof}

Our  next Proposition is a ready consequence.

\begin{proposition}\label{prop:M}
	If $d \geqslant 2$, then a.s.\ $[\mathrm{P}]$,
	\begin{equation}
		\lim_{N\to\infty} \max_S \frac{H_N}{N^2} =
		\frac{(\mathrm{E} q_0^+)^2 + (\mathrm{E} q_0^-)^2}{2} \in [0\,, \infty].
	\end{equation}
\end{proposition}

This result immediately
implies Proposition \ref{prop:maxH:bc} because the random
walk piece $\{S_i\}_{i=0}^{N-1}$ is equal to the argmax
of $S\mapsto H_N$ with probability  $(2d)^{-N}$. And therefore
\begin{equation}\label{eq:Z:maxH}
	Z_N (\beta)  \geqslant
	(2 d)^{- N} \exp \left( \beta \max_S \frac{H_N}{N}
	\right) .
\end{equation}
	
\begin{proof}[Proof of Proposition \ref{prop:M}]
	Owing to Lemma \ref{lem:HD}, we can decompose the
	maximum energy as
	\begin{equation}
		\max_S H_N = (Q_+^{\tmop{odd}})^2 + (Q_-^{\tmop{odd}})^2 +
		(Q_+^{\tmop{even}})^2 + (Q_-^{\tmop{even}})^2 .
  \end{equation}
	And one can check readily that
	the strong law of large numbers for i.i.d.\ nonnegative random 
	variables implies that a.s.\ $[\mathrm{P}]$,
	\begin{equation}
		\lim_{N\to\infty}\frac{Q_\varepsilon^p}{N / 2}
		= \mathrm{E} q^\varepsilon_0
		\quad
		\text{for all $\varepsilon\in\{-\,,+\}$ and
		$p\in\{\text{\rm odd}\,,\text{\rm even}\}$.}
	\end{equation}
	This completes the proof.
\end{proof}

Next we present a lower bound for $D_N$ in terms of four 
nonadjacent points. This bound
will play an important role in the proof of Theorem \ref{thm:square}.
It also will lead to an upper bound on the maximum energy
$\max_{_S} H_N$ in the case that $d = 1$.

\begin{lemma}\label{lem-Dxd1}
	If $d\geqslant 1$ and
	$\varepsilon, \varepsilon' \in \{-\,,+\}$ satisfy
	$\|x_{\varepsilon}^{\tmop{odd}} - 
	x_{\varepsilon'}^{\tmop{even}} \| \neq 1$,
	then
	\begin{equation}
		D_N \geqslant  \sum_{1\leqslant i < N :\, i \text{ \rm odd}} \min \left(
		Q_{\varepsilon}^{\tmop{odd}} q_i^{\varepsilon}\,,
		Q_{\varepsilon'}^{\tmop{even}} q_{i - 1}^{\varepsilon'} \right) .
	\end{equation}
\end{lemma}

\begin{proof}
	First of all, let us observe from the definition of $D$ that
	\begin{equation}
		D_N  \geqslant  Q_{\varepsilon}^{\tmop{odd}}
		\left( Q_{\varepsilon}^{\tmop{odd}} - \varepsilon
		Q^{x^{\tmop{odd}}_{\varepsilon}}_N \right) + 
		Q_{\varepsilon'}^{\tmop{even}}
		\left( Q_{\varepsilon'}^{\tmop{even}} - \varepsilon'
		Q^{x^{\tmop{even}}_{\varepsilon'}}_N \right).
	\end{equation}
	Next we note that
	\begin{equation}\begin{split}
		Q_{\varepsilon}^{\tmop{odd}} - \varepsilon
			Q^{x^{\tmop{odd}}_{\varepsilon}}_N 
			& \geqslant  \sum_{i \text{ odd} :\, S_i
			\neq x_{\varepsilon}^{\tmop{odd}}} q_i^{\varepsilon},
			\quad\text{and}\\
		Q_{\varepsilon'}^{\tmop{even}} - \varepsilon'
			Q^{x^{\tmop{even}}_{\varepsilon'}}_N & \geqslant
			\sum_{i \text{ even} :\,
			S_i \neq x_{\varepsilon'}^{\tmop{even}}} q_i^{\varepsilon'}.
	\end{split}\end{equation}
	If $i \in \{1, \ldots, N - 1\}$ is odd, then we necessarily have 
	either $S_{i - 1} \neq x_{\varepsilon'}^{\tmop{even}}$ or $S_i \neq
	x_{\varepsilon}^{\tmop{odd}}$. Therefore, the lemma follows.
\end{proof}

The following lemma will also be useful in our forthcoming analysis.

\begin{lemma}\label{lem:lwbH}
	For all $d \geqslant 1$, $\beta \geqslant 0$, $\varepsilon >
	0$, $N \geqslant 1$, and $q_0, \ldots, q_{N - 1} \in \mathbf{R}$:
	\begin{equation}
		\mathrm{P}_N^{\beta} \left\{
		\max_S H_N - H_N \geqslant \varepsilon N^2
		\right\} \leqslant  \mathrm{e}^{N[ \ln (2 d) - \beta \varepsilon]}.
	\end{equation}
\end{lemma}

\begin{proof}
	Because
	\begin{equation}
		\mathrm{P}_N^{\beta} \left\{
		\max_S H_N - H_N \geqslant \varepsilon N^2
		\right\}\leqslant  \frac{\exp \left( \frac{\beta}{N}  \left( \max_S H_N
		- \varepsilon N^2 \right) \right)}{Z_N (\beta)},
	\end{equation}
	the lemma follows from \eqref{eq:Z:maxH}.
\end{proof}

Now we conclude the proof of Theorem \ref{thm:square}. We introduce
\begin{equation}\begin{split}
	\Gamma &:= \min_{\substack{\varepsilon\in\{-,+\}\\
		p\in\{\text{even},\text{odd}\}}} (Q_{\varepsilon}^p)^2,\\
	\Lambda &:= \min_{\varepsilon, \varepsilon'\in\{-,+\}} 
		\sum_{0\leqslant i < N :\, i \text{ \rm odd}}
		\min \left( Q_{\varepsilon}^{\tmop{odd}} q_i^{\varepsilon}
		\,,
		Q_{\varepsilon'}^{\tmop{even}} q_{i - 1}^{\varepsilon'} \right).
\end{split}  \label{eq:Gamma}\end{equation}

Recall from \eqref{eq:gamma} and
\eqref{eq:lambda} the quantities $\gamma$ and $\lambda$.
Then a direct inspection reveals that
\begin{equation}
	\lim_{N\to\infty}\frac{\Gamma}{(N/2)^2}=\gamma,
	\qquad
	\lim_{N\to\infty}\frac{\Lambda}{(N/2)^2}=\lambda.
\end{equation}

For every fixed $\delta > 0$, let us consider the events
\begin{align}
	\mathcal{E}_N^{\delta} &:= \left\{ (1 + \delta) \frac{\Gamma}{N^2}
		\geqslant \frac{\gamma}{4} \text{ \ and \ } (1 + \delta) \frac{\Lambda}{N^2}
		\geqslant \frac{\lambda}{4} \right\},\label{def:EN}\\
	\mathcal{C}_{\alpha} &:= \left\{ \begin{array}{l}
		\varepsilon Q_N^{x_{\varepsilon}^p} \geqslant
			\dfrac{1 + \alpha}{2} Q_{\varepsilon}^p\text{ for all $\varepsilon=\pm$
			and $p=\text{odd}/\text{even}$}\\
		\|x_{\varepsilon}^{\tmop{odd}} -
			x_{\varepsilon'}^{\tmop{even}} \|= 1
			\text{ for all $\varepsilon$, $\varepsilon'=\pm 1$}
	\end{array} \right\},\label{eq:Calpha}
\end{align}
so that $\mathcal{C}_{\alpha}$ is the event that the points
$x_{\pm}^{\tmop{odd} / \tmop{even}}$ are adjacent and possess each a
proportion at least $(1 + \alpha) / 2$ of the available charge.
Note, in particular, that
\begin{equation}\label{eq:CinS}
	\mathcal{C}_\alpha\subseteq\mathcal{S}_\alpha,
\end{equation}
where the event $\mathcal{S}_\alpha$ was defined in
\eqref{eq:CinS}.

\begin{proof}[Proof of Theorem \ref{thm:square}]
	In accord with Cram\'er's theorem there exists $c_{\delta} > 0$ such that
	\begin{equation}
		\mathrm{P}( \mathcal{E}_N^{\delta}) \geqslant  1 - \exp
		\left( - c_{\delta} N \right)
		\qquad\text{for all $N \geqslant 1.$}
		\label{eq:pEN}
	\end{equation}
	Next we observe that if $\varepsilon
	Q_N^{x^p_{\varepsilon}} \leqslant (1 + \alpha) Q^p_{\varepsilon} / 2$ for
	some $p \in \{\tmop{odd}\,, \tmop{even}\}$
	and $\varepsilon \in \{-\,,+\}$, then
	\begin{equation}
		D_N \geqslant \left(\frac{1 -\alpha}{2}\right) 
		(Q^p_{\varepsilon})^2 \geqslant \left(\frac{1 -\alpha}{2}\right) \Gamma,
	\end{equation} 
	in accord with the definition \eqref{eq:DN} of $D_N$.
	If, on the other hand,
	$\|x_{\varepsilon}^{\tmop{odd}} - x_{\varepsilon'}^{\tmop{even}} \| \neq 1$
	for some $\varepsilon, \varepsilon' \in \{-\,,+\}$ then $D_N \geqslant
	\Lambda$ [Lemma \ref{lem-Dxd1}]. Therefore, we may apply
	Lemmas \ref{lem:HD} and \ref{lem:lwbH} in conjunction to deduce that
	the following holds almost surely on $\mathcal{E}_N^\delta$:
	\begin{equation}\begin{split}
		\mathrm{P}_N^{\beta} \left( \mathcal{C}_{\alpha}^c \right) & \leqslant 
			\mathrm{P}_N^{\beta} \left\{ 
			D_N \geqslant \min \left( \frac{1 - \alpha}{2} \cdot
			\Gamma\,, \Lambda \right) \right\}\\
		& \leqslant  \mathrm{P}_N^{\beta} \left\{ (1 + \delta) \frac{D_N}{N^2}
			\geqslant \min \left( \frac{1 - \alpha}{2} \cdot \frac{\gamma}{4}\,,
			\frac{\lambda}{4} \right) \right\}\\
		& \leqslant  \exp \left( N \left[ \ln (2 d) - \frac{\beta}{1 + \delta}
			\min \left( \frac{1 - \alpha}{2} \cdot \frac{\gamma}{4}\,,
			\frac{\lambda}{4} \right) \right] \right).
	\end{split}\label{eq:PC}\end{equation}
	This and \eqref{eq:CinS} together imply the result.
\end{proof}

Our next result estimates the maximum allowable
energy $\max_S H_N / N^2$ in the case that $d =1$.
It might help to recall that $\lambda$ was defined
in \eqref{eq:lambda}.

\begin{lemma}\label{lem:maxH:d1}
	If $d=1$, then 
	\begin{equation}
		\limsup_{N\to\infty} \frac{1}{N^2} \max_S H_N  \leqslant  \frac{%
		(\mathrm{E} q_0^+)^2 + (\mathrm{E} q_0^-)^2}{2} - \frac{\lambda}{4}
		\qquad\text{a.s. $[\mathrm{P}]$.}
	\end{equation}
	And, for all $\varepsilon \in \{-\,,+\}$,
	\begin{align}
		&\liminf_{N\to\infty} \frac{1}{N^2} \max_S H_N\\
		& \geqslant
			\frac{(\mathrm{E} q_0^+)^2 + (\mathrm{E} q_0^-)^2}{4} +
			\frac{a_4}{4} ( \mathrm{E} q_0^{\varepsilon})^2
			+ \left( \frac{\mathrm{E} q_0}{2} - \frac{\varepsilon a_2\mathrm{E}
			q_0^{\varepsilon} }{2}  \right)^2,
	\end{align}
	almost sure $[\mathrm{P}]$, where
	\begin{equation}
		a_k:=\left[ \mathrm{P}\{q_0\geqslant 0\}\right]^k
		+\left[ \mathrm{P}\{q_0<0\}\right]^k
		\qquad\text{for $k=2,4$}.
	\end{equation}
\end{lemma}
 
\begin{remark}
	In the case that $q_0$ has the Rademacher distribution
	[i.e., $\mathrm{P}\{q_0=\pm 1\}=1/2$], the preceding
	tells us that
	\begin{equation}
		\frac{19}{128} \leqslant
		\liminf_{N\to\infty} \frac{1}{N^2} \max_S H_N
		\leqslant \limsup_{N\to\infty}\frac{1}{N^2}
		\max_S H_N \leqslant \frac{7}{32}.
	\end{equation}
	[Note that $19/128\approx 0.1484375$ and
	$7/32\approx 0.21875$.]
\end{remark}

\begin{proof}[Proof of Lemma \ref{lem:maxH:d1}]
	We use the same notation as in the former proof. Since we have $d = 1$ it is
	not possible that $x_{\pm}^{\tmop{odd}}$ are adjacent to
	$x_{\pm}^{\tmop{even}}$. In view of Lemma \ref{lem-Dxd1} this implies that
	\begin{equation}
		(1 + \delta) \frac{D_N}{N^2} \geqslant  \frac{\lambda}{4}
		\qquad\text{for all $q \in \mathcal{E}^{\delta}_N$,}
	\end{equation}
	and hence $\max_S H_N/N^2$ is bounded above by
	\begin{equation}
		\frac{(Q_+^{\tmop{odd}})^2 +
		(Q_-^{\tmop{odd}})^2 + (Q_+^{\tmop{even}})^2 + 
		(Q_-^{\tmop{even}})^2}{N^2}
		- \frac{\lambda}{4(1+\delta)},
	\end{equation}
	 for every $q \in \mathcal{E}^{\delta}_N$. This yields
	 the first assertion of the lemma.
  
	We propose the following strategy in order to
	establish the asserted [asymptotic]
	lower bound on $N^{-2}\max_SH_N$: Choose
	and fix a sign $\varepsilon \in
	\{\pm\}$, and place odd monomers at positions $S_i = 1$ if $q_i \geqslant 0$,
	$S_i = - 1$ otherwise, and even monomers---whenever possible [that is, $S_{i - 1}
	= S_{i + 1}$]---at position $S_i = \pm 2$ if $\tmop{sign} (q_i) =
	\varepsilon$ and $S_i = 0$ otherwise. A computation, involving
	the strong law of large numbers, then shows that almost surely
	$[\mathrm{P}]$,
	\begin{equation}\begin{split}
		\lim_{N\to\infty}
			\frac{Q_N^{\pm 1}}{N} & = \pm \frac{\mathrm{E}
			q_0^{\pm}}{2},\\
		\lim_{N\to\infty}
			\frac{Q_N^{+ 2}}{N} & = \varepsilon \frac{\mathrm{E}
			q_0^ \varepsilon}{2}
			\left(\mathrm{P} \{q_0 \geqslant 0\}\right)^2,\\
		\lim_{N\to\infty}
			\frac{Q_N^{- 2}}{N} & =\varepsilon \frac{\mathrm{E}
			q_0^{\varepsilon}}{2} \left( \mathrm{P} \{
			q_0 < 0\}\right)^2,\quad\text{and}\\
		\lim_{N\to\infty}
			\frac{Q_N^0}{N} & =\frac{\mathrm{E} q_0}{2} - \varepsilon
			a_2\frac{\mathrm{E} q_0^{\varepsilon}}{2}.
	\end{split}\end{equation}
	This yields the lower bound.
\end{proof}

\subsection{Logarithmic range and bounded expectation of $|S_N |$
(Proofs of Theorems \ref{thm:range} and \ref{thm:ER})}

In this Section we prove Theorem \ref{thm:log} below and derive
Theorem \ref{thm:range} from it. 
We also present here a proof of Theorem \ref{thm:ER}. 

Given $N \geqslant 1$ and $L \geqslant 1$, define
\begin{equation}
	\bar{\mathsf{q}}_L  := \min_{\substack{\ell \geqslant L\\0\leqslant
	i < N - \ell}}\left( \frac{1}{\ell} \sum_{k = i}^{i
	+ \ell - 1} |q_k |\right).
	\label{eq:qL}
\end{equation}

\begin{theorem}[Logarithmic diameter]\label{thm:log} 
	\begin{enumerateroman}
	\item If $d \geqslant 1$ and $\mathrm{E} |q_0 | < \infty$, then
		for every $\beta \in \mathbf{R}$ and $\varepsilon > 0$ there 
		exists $c > 0$ such that for all sufficiently large integers
		$N\geqslant1$,
		\begin{equation}
			\mathrm{E}\left[ \mathrm{P}_N^{\beta} \left\{
			\frac{\tmop{Diam} \{S_i :\, 0\leqslant i <
			N\}}{\ln N} \geqslant c\right\}\right]
			\geqslant 1 - \exp \left( - c N^{1 -
			\varepsilon} \right).
			\label{eq:Diamslog}
		\end{equation}
	\item If $d \geqslant 2$, then for every $\alpha \in (0, 1)$,
		$N\geqslant 1$, $0\leqslant L\leqslant N$, and $\beta \in \mathbf{R}$,
		\begin{equation}\begin{split}
			&\mathrm{P}_N^{\beta} \left( \Big\{ \tmop{Diam}
			\{S_i :\, 0\leqslant i < N\}
			\geqslant L + 1 \Big\} \cap \mathcal{C}_{\alpha} \right) \\
			&\hskip1.3in
			\leqslant N^2 \exp \left( L \left[ \ln (2 d) -
			\frac{ 2 \beta \alpha\sqrt{\Gamma}}{N} 
			\,\bar{\mathsf{q}}_L \right] \right) .
			\label{eq:Diamilog}
		\end{split}\end{equation}
	\end{enumerateroman}
\end{theorem}

First we present a quick proof of Theorem \ref{thm:range}
that uses Theorem \ref{thm:log}. Then we establish
the latter result.

\begin{proof}[Proof of Theorem \ref{thm:range}]
	We apply \eqref{eq:Diamslog} and
	\eqref{eq:Diamilog} with $L := C \ln N$ to obtain
	all but part 2 immediately. And part 2 follows from Theorem~\ref{thm:square} and
	from Cram\'er's theorem, since
	\begin{equation}
		\mathrm{P} \left\{ \bar{\mathsf{q}}_L 
		\leqslant \mathrm{E} |q_0 | - \varepsilon \right\}
		\leqslant  N^2 \sup_{l \geqslant L} \mathrm{P} \left\{\frac{|q_1 | +
		\cdots + |q_l |}{l} \leqslant \mathrm{E} |q_0 | - \varepsilon \right\}
	\end{equation}
	decays more quickly than $N^{- K}$, provided that 
	$C$ is large enough.
\end{proof}

Our proof of Theorem \ref{thm:log} hinges on an analysis of the trajectory
of a certain portion of the polymer, conditional on 
the charges and the remaining portions of the polymer. We
begin with a Lemma that is useful for bounding the range of
the polymer from above. 

Choose and fix an integer $N\geqslant 1$, and let $I$ be a contiguous subset of $\{0\,, \ldots, N - 1\}$ with $|I| < N$. Given a realization of the polymer $S$ that satisfies $C_\alpha$ for some $0<\alpha<1$, we say that monomer $i\in\{0,\ldots,N-1\}$ is optimal when $S_i = x_{\tmop{sgn} (q_i)}^{\tmop{parity} (i)}$ (when $q_i=0$, monomer $i$ is optimal when $S_i \in \{ x_+^{\tmop{parity} (i)}, x_-^{\tmop{parity} (i)}\}$). By extension, we say that $S$ is nonoptimal on $I$ when none of the monomers $i\in I$ are optimal.
Define
\begin{equation}\begin{split}
	\mathcal{N}(I) &:= \left\{ S \text{ is nonoptimal on } I \right\},
		\quad\text{and}\\
	\mathcal{C}(I) &:= \left. \mathcal{C}_{\alpha} \cap \{S \text{ is optimal
		at the position(s) next to } I \right\},
\end{split}\end{equation}
where $\alpha \in (0\,, 1)$, and $\mathcal{C}_{\alpha}$ is the event defined
in \eqref{eq:Calpha}.

\begin{lemma}\label{lem:HSt}
	Let $N$, $\alpha$, and $I$ be fixed as above. Given a realization of $q$ and $S \in \mathcal{N}(I) \cap \mathcal{C}(I)$, define $\tilde{S}$ as follows:
	\begin{equation}\begin{split}
		\tilde{S}_i &= \left\{ \begin{array}{ll}
			S_i & \text{if } i \nin I,\\
		x^{\tmop{parity} (i)}_+ & \text{if } i \in I\ \&\ q_i \geqslant 0,\\
			x^{\tmop{parity} (i)}_- & \text{if } i \in I\ \&\ q_i < 0.
		\end{array} \right.
	\end{split}\end{equation}
	Then the trajectory
	$(\tilde{S}_i - \tilde{S}_0)_{i=0}^{N-1}$ is a possible
	realization of a simple random walk and
	\begin{equation}
		H_N ( \tilde{S}) - H_N (S)  \geqslant 2 \alpha 
		\mathop{\sum\sum}\limits_{\substack{
		\varepsilon\in\{-,+\}\\p\in\{\text{\rm odd},\text{\rm even}\}}}
		\left[ Q^p_{\varepsilon} \times \sum_{i \in I :\ i \equiv p}
		q_i^{\varepsilon} \right] .
	\end{equation}
\end{lemma}

\begin{proof}
	Because $S$ is optimal off $I$,
	$\tilde{S}$ is a simple random walk [but
	it might not start at the origin]. 
	
	Next we decompose $H_N(\tilde S)-H_N(S)$ as
	\begin{equation}\begin{split}
		&\sum_{x \in \{x^{\tmop{odd}
			\tmop{even}}_{\pm} \} \cup \{S_i;\, i \in I\}}
			 \left[\left( Q_N^x ( \tilde{S})
			\right)^2 - \left( Q_N^x (S) \right)^2\right]\\
		&= \sum_{x \in \{x^{\tmop{odd} / \tmop{even}}_{\pm} \} \cup \{S_i;\,
			i \in I\}} \left( Q_N^x ( \tilde{S}) + Q_N^x (S) \right) \left( Q_N^x (
			\tilde{S}) - Q_N^x (S) \right).
	\end{split}\end{equation}
	Now we observe that: (i) If $x = x^p_{\varepsilon}$, then
	\begin{equation}
		Q_N^x ( \tilde{S}) = Q_N^x (S) + \sum_{i \in I : i \equiv p\,,\,
		\varepsilon q_i > 0} q_i;
	\end{equation}
	and (ii) If $x=S_i$ for some $i\in I$, then
	\begin{equation}
		Q_N^x ( \tilde{S}) = Q_N^x (S) - \sum_{i \in I : S_i = x} q_i .
	\end{equation}
	Consequently, we can write
	\begin{equation}
		H_N ( \tilde{S}) - H_N (S) := T_1-T_2, \label{HStT1T2}
	\end{equation}
	where
	\begin{equation}
		T_1:= \mathop{\sum\sum}\limits_{\substack{
		\varepsilon\in\{-,+\}\\p\in\{\text{\rm odd},
		\text{\rm even}\}}} \left( 2
		Q_N^{x^p_{\varepsilon}} (S) + 
		\sum_{i \in I : i \equiv p, \varepsilon q_i> 0} q_i \right)
		\sum_{i \in I : i \equiv p, \varepsilon q_i > 0} q_i,
	\end{equation}
	and
	\begin{equation}
		T_2:=\sum_{x \in \{S_i;\, i \in I\}} \left( Q_N^x ( \tilde{S}) + Q_N^x
		(S) \right) \sum_{i \in I :\, S_i = x} q_i .  \label{eq:HSt1}
	\end{equation}
	Since $\varepsilon Q_N^{x^p_{\varepsilon}} (S) \geqslant (1 + \alpha) / 2
	Q^p_{\varepsilon} (S)$,
	\begin{equation}
		T_1\geqslant  (1 + \alpha) \mathop{\sum\sum}\limits_{
		\substack{\varepsilon\in\{-,+\}\\p\in\{\text{\rm odd},
		\text{\rm even}\}}}
		Q^p_{\varepsilon} \times \sum_{i \in I :\, i \equiv p} q_i^{\varepsilon} . 
	\label{eq:HSt2}
	\end{equation}
	Let us write, temporarily,
	\begin{equation}
		\mathcal{X}^{\tmop{odd}} :=\left\{S_i;\, i \in I \text{ odd} \right\}.
	\end{equation}
	Then clearly
	\begin{equation}\begin{split}
		&\sum_{x \in \mathcal{X}^{\tmop{odd}}}
			\left( Q_N^x ( \tilde{S}) + Q_N^x (S) \right)
			\sum_{i \in I :\, S_i = x} q_i\\
		& \leqslant  \sum_{\varepsilon\in\{-,+\}} \sum_{x \in
			\mathcal{X}^{\tmop{odd}}} \left( Q_N^x ( \tilde{S}) + Q_N^x (S)
			\right)^{\varepsilon} \left[ \sum_{i \in I : S_i = x} q_i
			\right]^{\varepsilon}\\
		& \leqslant  \sum_{\varepsilon\in\{-,+\}} \max_{x \text{ odd}:\,
			x \neq x^{\tmop{odd}}_{\pm}} \left( Q_N^x ( \tilde{S}) + Q_N^x (S)
			\right)^{\varepsilon} \times \sum_{i \in I:\,i\text{ odd}}
			q_i^{\varepsilon}\\
		& \leqslant  (1 - \alpha) \sum_{\varepsilon\in\{-,+\}}
			Q_{\varepsilon}^{\tmop{odd}} 
			\sum_{i \in I:\,i \text{ odd}} q_i^{\varepsilon};
	\end{split}\end{equation}
	the last line is valid because, whenever 
	$x \neq x^{\tmop{odd}}_{\pm}$ is odd, the quantities $Q_N^x (
	\tilde{S})$ and $Q_N^x (S)$ both lie in the interval $[- \frac12
	(1 - \alpha) Q_-^{\tmop{odd}}, \frac12(1 - \alpha)  Q_+^{\tmop{odd}}]$.
	It follows that
	\begin{equation}
		T_2 \leqslant  (1 - \alpha)
		\mathop{\sum\sum}\limits_{\substack{
		\varepsilon\in\{-,+\}\\
		p\in\{\text{\rm odd},\text{\rm even}\}}} Q_{\varepsilon}^p 
		\sum_{i \in I:\,i \equiv p}
		q_i^{\varepsilon} .  \label{eq:HSt3}
	\end{equation}
	The claims follows from \eqref{HStT1T2}, \eqref{eq:HSt2},
	and \eqref{eq:HSt3}.
\end{proof}

\begin{proof}[Proof of Theorem \ref{thm:log}]
	We begin by deriving \eqref{eq:Diamslog}.
	
	With probability exponentially close to one [as $N \rightarrow \infty$],
	the total charge of the polymer satisfies
	\begin{equation}\label{eq:good:q}
		\sum_{i = 0}^{N - 1} |q_i |  \leqslant  2 N \mathrm{E} |q_0 |.
	\end{equation}
	Therefore, by conditioning, we may [and will] assume that the
	charges satisfy the former inequality. 
	
	Because the $q$'s satisfy \eqref{eq:good:q},
	 it follows that if we modify a single position $S_i$ of the polymer,
	then we change $H_N(S)$ by at most
	$8 N \mathrm{E} |q_0 | \times |q_i |$. 
	Consequently, 
	\begin{equation}\begin{split}
		&\mathrm{P}_N^{\beta} \left( \left. \begin{array}{l}
			\left. \tmop{Diam} \{S_i :\, N_1 \leqslant i < N_2 \right\}\\
			\geqslant (N_2 - N_1 - 1) / 2
			\end{array} \right| S_0, \ldots S_{N_1 - 1}, S_{N_2}, \ldots, S_{N - 1}
			\right)\\
		&\hskip1.3in\geqslant  \frac{1}{(2 d)^{N_2 - N_1}}
			\exp \left( - 8| \beta | \mathrm{E} |q_0 |
			\sum_{N_1 \leqslant i < N_2} |q_i | \right),
	\end{split}\label{eq:Diam1}\end{equation}
	almost surely for every  $N_1 < N_2$ in
	$\{0, \ldots, N\}$.
	Given $L\in\{1\,,\ldots,N-1\}$, define
	\begin{equation}
		\mathcal{K}_L := \left\{ k \in \{1, \ldots, [N / L]\}: \sum_{(k - 1) L
		\leqslant i < kL} |q_i | \leqslant 2 L \mathrm{E} |q_0 | \right\} .
	\end{equation}
	Then, \eqref{eq:Diam1} leads to the bound
	\begin{equation}\begin{split}
		&\mathrm{P}_N^{\beta} \left\{
			\max_{k\in\mathcal{K}_L}
			\tmop{Diam} \{S_i :\, (k - 1) L < i \leqslant kL\} <
			\frac{L-1}{2} \right\}\\
		&\hskip3.2in\leqslant  \left( 1 - a^L \right)^{
			|\mathcal{K}_L|},
	\end{split}\end{equation}
	where $a := \exp \left( - 16| \beta | ( \mathrm{E} |q_0 |)^2 \right) / (2
	d)$. Now we choose $L$ judiciously; namely, we let
	 $L := L_N := [- \varepsilon \ln (N) / \ln (a)]$---so that $a^L
	/ N^{- \varepsilon} \rightarrow 1$ as $N\to\infty$---in order to deduce
	the following:
	\begin{equation}\begin{split}
		&\mathrm{E} \left[ \mathrm{P}_N^{\beta}
			\left\{\tmop{Diam} \{S_i:\, i <
			N\}< \frac{L-1}{2} \right\}\right]\\
		& \leqslant  \mathrm{P} \left\{ \sum_{i = 1}^N
			|q_i | > 2 N \mathrm{E} |q_0 | \right\} + \mathrm{P} \left\{
			|\mathcal{K}_L |\leqslant \frac{N}{2 L} \right\}
			+ \left( 1 - N^{- \varepsilon} \right)^{N / (2 L)}.
	\end{split}\end{equation}
	This yields \eqref{eq:Diamslog}.
  
	We prove \eqref{eq:Diamilog} next. 
	
	If $\mathcal{C}_{\alpha}$
	holds and $S$ has $L$ consecutive nonoptimal monomers, 
	then we can find a contiguous $I \subset \{0\,, \ldots, N - 1\}$
	such that $L\leqslant |I|<N$
	and $S \in \mathcal{N}(I) \cap \mathcal{C}(I)$. There are
	not more than $N^2$ corresponding choices for such an $I$. Therefore,
	\begin{equation}\begin{split}
		&\mathrm{P}_N^{\beta} \left( \left\{ \begin{array}{l}
			S \text{ has } L \text{ consecutive}\\
			\text{nonoptimal monomers}
			\end{array} \right\} \cap \mathcal{C}_{\alpha} \right)\\
		&\hskip1.5in\leqslant  N^2
			\times \sup_{I \tmop{contiguous} : \left| I \right| \geqslant L}
			\mathrm{P}_N^{\beta} \left( \mathcal{N}(I) \cap
			\mathcal{C}(I) \right) .
	\end{split}\label{eq:claimL}\end{equation}
	Consider such a contiguous set $I$. Every $S \in \mathcal{N}(I) \cap
	\mathcal{C}(I)$ gets mapped to $\bar{S} := \tilde{S} - \tilde{S}_0 \in
	\mathcal{C}(I)$, and no more than $(2 d)^{\left| I \right|}$ choices of $S$
	yield the same $\bar{S}$. In addition, Lemma \ref{lem:HSt} and
	the definition \eqref{eq:Gamma} of $\Gamma$ together tell us that
	\begin{equation}\begin{split}
		H_N ( \tilde{S}) - H_N (S) 
			& \geqslant 2 \alpha \mathop{\sum\sum}\limits_{\substack{
			\varepsilon\in\{-,+\}\\p\in\{\text{\rm odd},\text{\rm even}\}}}
			\left[ \sqrt{\Gamma} \times
			\sum_{i \in I :\,i \equiv p} q_i^{\varepsilon} \right]\\
		&= 2 \alpha \sqrt{\Gamma}  \sum_{i \in I} |q_i |.
		\label{eq:DH}
	\end{split}\end{equation}
	Therefore,
	\begin{align}\nonumber
		&\mathrm{P}_N^{\beta} \left(
			\mathcal{N}(I) \cap \mathcal{C}(I) \right) \nonumber\\
		& \leqslant  \exp \left( - 2 \beta \alpha \frac{\sqrt{\Gamma}}{N} 
			\sum_{i \in I} |q_i | \right) \frac{1}{Z_N^{\beta}} \sum_{S \in
			\mathcal{N}(I) \cap \mathcal{C}(I)} \exp \left( - \beta H_N ( \tilde{S}) / N
			\right) \nonumber\\
		& \leqslant  \exp \left( - 2 \beta \alpha \frac{\sqrt{\Gamma}}{N}
			\sum_{i \in I} |q_i | \right) (2 d)^{\left| I \right|} 
			\mathrm{P}_N^{\beta} \left( \mathcal{C}(I) \right)
			\label{eq:PNI1}\\
		& \leqslant  \exp \left( |I| \times \left[ \ln (2 d) - 2 \beta \alpha
			\frac{\sqrt{\Gamma}}{N} \bar{\mathsf{q}}_L \right] \right),
			\label{eq:PNI}
	\end{align}
	owing to the definition of $\bar{\mathsf{q}}_L$. The claim follows
	from this and \eqref{eq:claimL}.
\end{proof}

\begin{proof}[Proof of Theorem \ref{thm:ER}]
	The proof is similar to the proof of Theorem \ref{thm:log}.
	
	Recall that $R_{\alpha}^N$ was defined in \eqref{def:R},
	and define $I :=\{0\,, \ldots, r - 1\}$ for some fixed $1\leqslant r < N$.
	Then, we may use \eqref{eq:PNI1}
	and the obvious fact that $\mathrm{P}_N^\beta(\mathcal{C}(I))
	\leqslant 1$ in order to deduce that
	\begin{equation}\begin{split}
		\mathrm{P}_N^{\beta} \left( \left\{
			R_{\alpha}^N = r \right\} \cap \mathcal{C}_{\alpha}
			\right) &= 
			\mathrm{P}_N^{\beta} \left( \mathcal{N}(I) \cap \mathcal{C}(I)
			\right)\\
		& \leqslant  \exp \left( - 2 \beta \alpha \frac{\sqrt{\Gamma}}{N}
			\sum_{i=0}^{r-1} |q_i | \right) (2 d)^r.
	\end{split}\end{equation}
	Next, we choose and fix an arbitrary $\delta > 0$.,
	and recall from \eqref{def:EN}
	the event $\mathcal{E}_N^{\delta}$. Then
	almost surely on $\mathcal{E}_N^\delta$,
	\begin{equation}
		\mathrm{P}_N^{\beta} \left( \left\{R_{\alpha}^N = r\right\} 
		\cap \mathcal{C}_{\alpha}
		\right) \leqslant  (2 d)^r \exp \left( - 2 \beta \alpha
		\sqrt{\frac{\gamma}{1 + \delta}} \sum_{i=0}^{r-1} 
		|q_i | \right).
	\end{equation}
	Because $R_\alpha^N\leqslant N$, it follows that
	\begin{align} \nonumber
		\mathrm{E}\left[\mathrm{E}_N^{\beta} \left( R_{\alpha}^N \right)\right]
			& \leqslant  N \left[1-\mathrm{P} (\mathcal{E}_N^{\delta}) \right]+ 
			N \mathrm{E}\left[1-\mathrm{P}_N^{\beta}
			\left( \mathcal{C}_{\alpha} \right) \right]\\
		&\hskip.7in+ \sum_{r=0}^{N-1} r \mathrm{E} \left[
			\tmmathbf{1}_{\mathcal{E}^{\delta}_N}  \mathrm{P}_N^{\beta}
			\left(\left\{ R_{\alpha}^N = r\right\}
			\cap \mathcal{C}_{\alpha} \right) \right]\\
		&\leqslant o(1)+\sum_{r=1}^{N-1} r (2 d)^r
			\mathrm{E}\left[\exp \left( - 2 \beta \alpha
			\sqrt{\frac{\gamma}{1 + \delta}} \sum_{i=0}^{r-1} 
			|q_i | \right)\right], \nonumber
	\end{align}
	as $N\to\infty$; see \eqref{eq:pEN} and \eqref{eq:PC}.	
	Define
	\begin{equation}
		\rho_{\delta} := 2 d \mathrm{E} \left[ \exp \left( - \beta \alpha
		\sqrt{\frac{\gamma}{1 + \delta}} |q_0 | \right) \right].
	\end{equation}
	Since $\lim_{\delta\downarrow 0}\rho_\delta=\rho$,
	it follows that $\rho_\delta<1$ for all $\delta>0$ sufficiently small.
	And hence, for all $\delta>0$ sufficiently small,
	\begin{equation}
		\limsup_{N\to\infty}
		\mathrm{E}\left[\mathrm{E}_N^{\beta} \left( R_{\alpha}^N \right)\right]
		\leqslant \sum_{r=1}^\infty r \rho_\delta^r=
		\frac{\rho_\delta}{(1-\rho_\delta)^2}.
	\end{equation}
	Let $\delta \rightarrow 0$ to finish.
\end{proof}

\subsection{On the annealed measure}

Our analysis of the quenched measure can be adapted with no difficulty,
and with some simplifications, to study also the annealed measure. 
Here we prove only Proposition \ref{prop:bca}.

\begin{proof}[Proof of Proposition \ref{prop:bca}]
	We know already from the analog of Theorem \ref{thm:D} that $\tilde{\beta}_c
	\geqslant 1$. Therefore, it suffices to prove that 
	$\mathrm{E}Z_N(1)=\infty$. Let 
	\begin{equation}
		\nu := \nu(N):= \lceil N/2 \rceil.
	\end{equation}
	Because $\mathrm{P}\{ L_N^\star = \nu\}\geqslant (2d)^{-N}$
	for all $N$ sufficiently large,
	it follows immediately from properties of the normal distribution
	that
	\begin{equation}\begin{split}
		\mathrm{E} Z_N (1) &\geqslant (2 d)^{- N} \mathrm{E}
			\left[ \exp\left\{ \frac{
			\left( q_1 + \cdots + q_\nu \right)^2}{N} \right\}\right]\\
		&= (2 d)^{- N} \mathrm{E} \mathrm{e}^{\nu q_0^2/N}.
	\end{split}\end{equation}
	And the latter quantity is infinite because $\nu / N \geqslant 1 / 2$. 
\end{proof}

\subsection{The influence of a pulling force}\label{sec:pulling:proof}

First we justify our claim that the results of
Theorem \ref{thm:D}, Proposition \ref{prop:F}, and Theorem \ref{thm:fo} 
continue to hold [up to a modification of the notation]. Basically,
this is so because Lemma \ref{lem:moments} is the only place where we
explicitly used the fact that $S$ is the simple symetric random walk.
Now the new measure $\mathrm{P}_\lambda$ has the following property:
\begin{equation}
	\mathrm{P}_{\lambda} \{S_k = 0\} = \frac{\mathrm{P}
	\{S_k = 0\}}{(\mathrm{E}\exp (\lambda \cdot S_1))^k},
	\label{eq:lreturn}
\end{equation}
with $\mathrm{E} \exp (\lambda \cdot S_1) > 1$ whenever
$\lambda \neq 0$. Therefore, 
the local time at the origin satisfies $\mathrm{E}_{\lambda} L_N^0
\leqslant \mathrm{E} L_N^0$. This is enough for concluding that even the
statement of Lemma \ref{lem:moments} continues to hold
when we replace $\mathrm{E}$ with $\mathrm{E}_{\lambda}$ for
$\lambda \in \mathbf{R}^d$. Next we
prove Theorem \ref{thm:pulling}.

\begin{proof}[Proof of Theorem \ref{thm:pulling}]
	We begin with the proof of \eqref{eq:bcg}: We
	know already that $\beta_c (\lambda) \geqslant 1 / \kappa$.
	[Theorem \ref{thm:D}]. Let us choose and fix some
	$\varepsilon > 0$. Then we can write
	\begin{equation}\label{eq:2terms}\begin{split}
		&Z_N (\beta\,, \lambda) \\
		&\hskip.2in= Z_N^{(1-\varepsilon)/ (2 \kappa \beta)}
			(\beta\,, \lambda) + \mathrm{E}_{\lambda} \left[ \left. 
			\text{\rm e}^{\beta H_N/N}\tmmathbf{1}_{\mathcal{A}(N)} 
			\,\right|\, q_0, q_1, \ldots, q_{N -
			1} \right],
	\end{split}\end{equation}
	where $Z_N^{\varepsilon} (\beta, \lambda)$ is defined
	by adapting \eqref{eq:Zeps}---in the obvious way---to
	the new reference measure $\mathrm{P}_{\lambda}$,
	and $\mathcal{A}(N)$ denotes the following event:
	\begin{equation}
		\mathcal{A}(N) := \left\{ \frac{L_N^{\star}}{N} >
		\frac{1-\varepsilon}{2 \kappa \beta}\right\}.
	\end{equation}
	We know from Theorem \ref{thm:D}
	that $Z_N^{(1-\varepsilon)/ (2 \kappa \beta)} (\beta, \lambda)
	\rightarrow {\rm e}^{\beta}$ in probability as $N\to\infty$. 
	Next we consider the second term in \eqref{eq:2terms}. 
	
	Because
	\begin{equation}\begin{split}
		&\mathrm{E}_{\lambda} \left[ \left. \exp \left( \frac{\beta}{N} H_N \right)
			\tmmathbf{1}_{\mathcal{A}(N)} 
			\,\right|\, q_0, q_1, \ldots, q_{N
			- 1} \right]\\
		& \hskip1in\leqslant  \exp \left( \beta N \max_S \frac{H_N}{N^2}
			\right) \mathrm{P}_{\lambda} \left\{
			\frac{L_N^{\star}}{N} > \frac{1-\varepsilon}{2
			\kappa \beta}\right\},
	\end{split}\end{equation}
	it follows that
	\begin{equation}
		\digamma_{\lambda} (\beta) \leqslant  \max \left( \beta
		\frac{(\mathrm{E} q_0^+)^2 + 
		(\mathrm{E} q_0^-)^2}{2} - I_{\lambda} \left(
		\frac{1- \varepsilon}{2 \kappa \beta}  \right)\,,\,0 \right),
	\end{equation} 
	where $I_{\lambda}$ was defined in \eqref{eq:Il}. 
	
	We conclude the proof by establishing a lower
	bound for $I_{\lambda}$. 
	
	According to \eqref{eq:lreturn},
	\begin{equation}
		\mathrm{E}_\lambda L_{\infty}^0 \leqslant
		\frac{1}{1-1/\mathrm{E} (\exp (\lambda \cdot S_1)) }.
	\end{equation}
	By the strong Markov property, $L_{\infty}^0$ 
	has a  geometric distribution with parameter
	\begin{equation}
		p := \mathrm{P}_\lambda\left\{ S_i\neq 0\text{ for all 
		$i\geqslant 1$}\right\},
	\end{equation}
	and therefore $p \geqslant 1-1/\mathrm{E} (\exp (\lambda \cdot S_1)) $.
	It follows that
	\begin{equation}
		\mathrm{P}_{\lambda} \left\{
		L^0_N \geqslant \alpha N \right\} \leqslant 
		\mathrm{P}_{\lambda} \left\{ L^0_{\infty} \geqslant \alpha N
		\right\} = (1 - p)^{\lceil\alpha N\rceil- 1},
	\end{equation}
	and consequently
	\begin{eqnarray}
		I_{\lambda} (\alpha) \geqslant \alpha \ln \mathrm{E} (\exp (\lambda
		\cdot S_1)).
	\end{eqnarray}
	The conclusion \eqref{eq:bcg} is immediate. 
	Now we address the opposite bound \eqref{eq:bci}. 
	
	If $e \in \mathbf{Z}^d$ has norm 1, then
	\begin{eqnarray}
		\mathrm{P}_{\lambda} \left\{S_{k + 1} - S_k = e\right\}
		\geqslant \frac{\exp \left( - 2 \left\|
		\lambda \right\|_{\infty} \right)}{2 d},
	\end{eqnarray}
	whence
	\begin{equation}
		Z_N (\beta\,, \lambda) \geqslant\exp \left( \beta N \max_S
		\frac{H_N}{N^2} \right) \times \left( \frac{\exp \left( - 2 \left\|
		\lambda \right\|_{\infty} \right)}{2 d} \right)^N.
  	\end{equation}
	Consequently, \eqref{eq:bci} follows from
	Proposition \ref{prop:M} when $d \geqslant   2$; 
	and \eqref{eq:bci} follows
	from Lemma \ref{lem:maxH:d1} when $d = 1$. 
	
	Finally, we prove that $\beta_c$ is locally Lipschitz.
	
	The density
	\begin{equation}
		\left. \frac{\mathd \mathrm{P}_{\lambda + \mu}}{\mathd
		\mathrm{P}_{\lambda}} \right|_{\sigma (q, S_0, \ldots, S_k)} :=
		\frac{\exp (\mu \cdot S_k)}{(
		\mathrm{E}_{\lambda} \exp (\mu \cdot S_1))^k}
	\end{equation}
	is bounded from above and below respectively by 
	$\exp( \pm 2 k\| \mu\|_1)$. 
	Therefore, for all $\beta \in \mathbf{R}$,
	\begin{equation}\begin{split}
		Z_N (\beta\,, \lambda + \mu) & \geqslant Z_N (\beta\,, \lambda) \exp \left(
			- 2 N \left\| \mu \right\|_{\infty} \right),\\
		\digamma_{\lambda + \mu} (\beta) & \geqslant \digamma_{\lambda} (\beta)
			- 2 \left\| \mu \right\|_\infty.
	\end{split}\end{equation}
	This proves the claim when one chooses
	\begin{equation}
		\beta  >  \beta_c (\lambda) + 2 \frac{\left\| \mu
		\right\|_{\infty}}{\digamma_{\lambda}' (\beta_c (\lambda))},
	\end{equation} 
	for which $\digamma_{\lambda + \mu} (\beta) > 0$ 
	[thanks to the convexity
	of $\digamma_{\lambda + \mu}$]. The lower bound on $\digamma_{\lambda}'
	(\beta_c (\lambda))$ comes from the generalization of \eqref{eq:prop:min} in
	Theorem \ref{thm:fo}.
\end{proof}

\appendix 

\section{The local times of the random walk}
In this appendix we collect some facts about the local times
of the simple random walk $\{S_i\}_{i=0}^\infty$ on 
$\mathbf{Z}^d$.
Recall that the local time at $x$ of the walk is
denoted by the process $\{L_N^x\}_{N=1}^\infty$, and is
defined by $L_N^x := \sum_{0\leqslant i<N} \mathbf{1}_{\{S_i=x\}}$.

\begin{lemma}\label{lem:LT0}
	There exists a continuous nondecreasing
	function $I:(0\,,1/2)\to(0\,,\infty)$
	such that
	\begin{equation}
		\lim_{N\to\infty} \frac1N\ln
		\mathrm{P}\left\{ L_N^0 > \varepsilon N\right\}
		=-I(\varepsilon)\qquad\text{for all 
		$\varepsilon\in(0\,,1/2)$}.
	\end{equation}
\end{lemma}

In fact, the limit exists for all $\varepsilon>0$.
But the additional gain in generality
is uninteresting because
$\mathrm{P}\{L_N^0>\varepsilon N\}=0$---whence
$I(\varepsilon)=\infty$---when $\varepsilon\geqslant 1/2$, since the simple walk
on $\mathbf{Z}^d$ has period 2.

\begin{proof}
	Let $\tau_0:=0$ and for $k\geqslant 1$ define
	$\tau_k$ to be the $k$th return time
	to the origin by the random walk; that is,
	$\tau_k:=\min\{j>\tau_{k-1}:\, S_j=0\}$.
	It is easy to see that $L_N^0>\varepsilon N$
	if and only if $\tau_{[\varepsilon N]}<N$. According 
	to a result of Jain and Pruitt \cite[Theorem 2.1]{JP87},
	\begin{equation}
		\lim_{N\to\infty}\frac1N\mathrm{P}
		\left\{ \tau_{[\varepsilon N]}<N\right\}
		=-R\left( g^{-1}(1/\varepsilon)\right)
		\quad\text{for all $\varepsilon\in(0\,,1/2),$}
	\end{equation}
	where $R$ is continuous, $g$ is continuous and
	strictly decreasing, and both are defined as follows:
	\begin{equation}
		g(u) := -\frac{\varphi'(u)}{\varphi(u)}\quad
		\text{and}\quad
		R(u):=-\ln\varphi(u)-ug(u),
	\end{equation}
	where $\varphi(u)=\mathrm{E}\exp(-u\tau_1)$.
	This implies our lemma
	with
	\begin{equation}
		I(\varepsilon):=(R\circ g^{-1})(1/\varepsilon)
		\quad\text{for all $\varepsilon\in(0\,,1/2).$}
	\end{equation}
	Let us also mention that
	$\varphi$  can be computed, in a standard way, by
	appealing to excursion theory \cite[Lemma 2.1]{K}. The end-result is that
	\begin{equation}
		\varphi(u) = \frac{1}{(2\pi)^d}\int_{(-\pi,\pi)^d}
		\frac{{\rm d}\xi}{1- G(\xi){\rm e}^{-u}},
	\end{equation}
	where  $G(\xi) := d^{-1}\sum_{j=1}^d\cos(\xi\cdot
	\mathbf{e}_j)$ for the $d$ standard basis vectors
	$\{\mathbf{e}_j\}_{j=1}^d$ of $\mathbf{R}^d$.
	We omit the details of this standard computation.
\end{proof}

Recall that $L_N^\star:=\sup_{x\in\mathbf{Z}^d}L_N^x$ denotes the
maximum local time.

\begin{lemma}\label{lem:LTstar}
	For every fixed $x\in\mathbf{Z}^d$, $L_N^x$ is stochastically
	smaller than $L_N^0$. Therefore, for the same 
	function $I$ as in Lemma \ref{lem:LT0},
	\begin{equation}
		\lim_{N\to\infty}\frac1N\mathrm{P}\left\{
		L_N^\star >\varepsilon N\right\}=-I(\varepsilon)
		\qquad\text{for all $\varepsilon\in(0\,,1/2)$}.
	\end{equation}
\end{lemma}

\begin{proof}
	Recall that the assertion about stochastic monotonicity
	is simply that 
	$\mathrm{P} \{ L_N^x >a \}
	\leqslant \mathrm{P} \{ L_N^0>a \}$
	for all $a\in\mathbf{Z}^d$. 
	This is a ready consequence of the strong Markov
	property [applied at the first hitting time of the origin].
	Because
	\begin{equation}
		\mathrm{P}\left\{ L_N^0>a\right\}
		\leqslant
		\mathrm{P}\left\{ L_N^\star > a\right\}
		\leqslant\sum_{\substack{x\in\mathbf{Z}^d:\\
		\|x\|_1\leqslant n}}
		\mathrm{P}\left\{ L_N^x>a\right\},
	\end{equation}
	stochastic monotonicity implies that for all $N\geqslant 1$,
	\begin{equation}\label{eq:SM}
		\mathrm{P}\left\{ L_N^0>a\right\}
		\leqslant
		\mathrm{P}\left\{ L_N^\star > a\right\}
		\leqslant (2N)^d
		\mathrm{P}\left\{ L_N^0>a\right\}.
	\end{equation}
	Therefore, Lemma \ref{lem:LT0} finishes the proof.
\end{proof}

{\noindent}\tmtextbf{Acknowledgements.} We are grateful to Amine Asselah and
Bernard Derrida for references and helpful discussions. It is a pleasure as
well to thank Nicolas Petrelis for his helpful comments on the first-order phase
transition, and Dmitry Ioffe for suggesting that we include the
material in Section~\ref{sec:pulling}. We are also grateful to Romain Abraham for suggesting the present formulation of part 3 of Theorem \ref{thm:D}.

\begin{small}

\bibliographystyle{plain}
\bibliography{biblio}

 \vskip1cm

{\noindent}\tmtextbf{Yueyun Hu.} D\'epartement de Math\'ematiques,
Universit\'e Paris XIII, 99 avenue J-B Cl\'ement, F-93430 Villetaneuse,
France, {\hspace{1em}}\tmtextit{Email:}
\tmtexttt{yueyun@math.univ-paris13.fr}\\

{\noindent}\tmtextbf{Davar Khoshnevisan.} Department of Mathematics,
University of Utah, 155 South 1440 East, JWB 233, Salt Lake City, Utah
84112-0090, USA,\\
\tmtextit{Email:} \tmtexttt{davar@math.utah.edu}\\

{\noindent}\tmtextbf{Marc Wouts.} D\'epartement de Math\'ematiques,
Universit\'e Paris XIII, 99 avenue J-B Cl\'ement, F-93430 Villetaneuse,
France, {\hspace{1em}}\tmtextit{Email:} \tmtexttt{wouts@math.univ-paris13.fr
}\\
\end{small}

\end{document}